\documentclass[12pt, reqno]{amsart}
\usepackage{amsmath, amsthm, amscd, amsfonts, amssymb, graphicx, color, mathrsfs}
\usepackage[bookmarksnumbered, colorlinks, plainpages]{hyperref}
\usepackage[all]{xy} 
\usepackage{slashed}

\newtheorem{theorem}{Theorem}[section]
\newtheorem{lemma}[theorem]{Lemma}

\newtheorem{corollary}[theorem]{Corollary}
\theoremstyle{definition}

\theoremstyle{remark}
\newtheorem{remark}[theorem]{Remark}
\numberwithin{equation}{section}

\begin{document}
\setcounter{page}{1}

\title[Fourier multipliers  in Triebel-Lizorkin spaces]{Fourier multipliers for  Triebel-Lizorkin spaces on compact Lie groups}

\author[D. Cardona]{Duv\'an Cardona}
\address{
  Duv\'an Cardona S\'anchez:
  \endgraf
  Department of Mathematics: Analysis, Logic and Discrete Mathematics
  \endgraf
  Ghent University, Belgium
  \endgraf
  {\it E-mail address} {\rm duvanc306@gmail.com,\, Duvan.CardonaSanchez@ugent.be }
  }

\author[M. Ruzhansky]{Michael Ruzhansky}
\address{
  Michael Ruzhansky:
  \endgraf
  Department of Mathematics: Analysis, Logic and Discrete Mathematics
  \endgraf
  Ghent University, Belgium
  \endgraf
 and
  \endgraf
  School of Mathematical Sciences
  \endgraf
  Queen Mary University of London
  \endgraf
  United Kingdom
  \endgraf
  {\it E-mail address} {\rm Michael.Ruzhansky@ugent.be}
  }

\thanks{The authors are supported by the FWO Odysseus 1 grant G.0H94.18N: Analysis and Partial Differential Equations and by the Methusalem programme of the Ghent University Special Research Fund (BOF) (Grant number 01M01021). MR is also supported in parts by the EPSRC grant
EP/R003025/2.
}


\begin{abstract} 
 We  investigate the boundedness of Fourier multipliers on a compact Lie group when acting on Triebel-Lizorkin spaces. Criteria are given in terms of the H\"ormander-Mihlin-Marcinkiewicz condition. In our analysis, we use  the difference structure of the unitary dual of a compact Lie group. Our results cover the sharp H\"ormander-Mihlin theorem on Lebesgue spaces  and also other historical results on the subject. 
\end{abstract} 

\maketitle

\tableofcontents
\allowdisplaybreaks

\section{Introduction} 

Let $G$ be a compact Lie group. In this work we  study  sufficient  conditions for the boundedness of Fourier multipliers on the Triebel Lizorkin spaces $F^{r}_{p,q}(G)$ in terms of the H\"ormander-Mihlin condition on their symbols. The  Littlewood-Paley theorem states that $L^p(G)\equiv F^{0}_{p,2}(G),$ so that in view of the classical results of H\"ormander-Mihlin type (see H\"ormander \cite{Hormander1960} and Mihlin \cite{Mihlin} for instance), Triebel-Lizorkin spaces are a good substitute of $L^p$-spaces, when considering smoothness of distributions in different scales (see Triebel \cite{Triebel1983,Triebel2006} and \cite{NurRuzTikhBesov2015,NRT} for details).  

The problem of finding conditions for the boundedness of Fourier multipliers on compact Lie groups started a long of time ago with the study of the theory of Fourier series of periodic functions on $[0,1],$ (functions on the torus $\mathbb{T}=\mathbb{R}/\mathbb{Z}$). Indeed, 
when classifying the boundedness of multipliers of  the Fourier series, 
\begin{equation}
    Af(x):=\sum_{\xi\in \mathbb{Z}}e^{2\pi i x\cdot \xi}\sigma(\xi)\widehat{f}(\xi),\quad f\in L^{1}(\mathbb{T}),\,\, \widehat{f}(\xi):=\int\limits^{1}_{0}e^{-2\pi i x\cdot \xi}f(x)dx,\quad \,
\end{equation}   it was observed by Marcinkiewicz in his classical 1939's work \cite{Marc} that the condition
\begin{equation}\label{Marc1}
\sup_{\xi\in\mathbb{Z}}|\sigma(\xi) |+ \sup_{j\in\mathbb{N}_0}\sum_{2^{j-1}\leq |\xi|<2^{j}}|\sigma(\xi+1)-\sigma(\xi)|<\infty,
\end{equation}assures the existence of a bounded extension of $A$ on $L^p(\mathbb{T}),$ $1<p<\infty.$ Denoting the difference operator on the lattice $\mathbb{Z},$ by  $\Delta \sigma:=\sigma(\cdot+1)-\sigma,$ and by $\Delta^k,$ $k\in \mathbb{N}_0,$ its successive iterations, the Marcinkiewicz condition \eqref{Marc1} is satisfied by any sequence $(\sigma(\xi))_{\xi\in \mathbb{Z}^n}$ such that
\begin{equation}\label{Mac2}
    |\Delta^k \sigma(\xi)|\lesssim_{k} |\xi|^{-k},\quad \,\xi\neq 0,\,k=0,1,
\end{equation} which should be, in principle, more easier to verify that \eqref{Marc1}. Another generalisation of  Marcinkiewicz's criterion  for  multipliers of the Fourier transform on $\mathbb{R}^n,$ was done by Mihlin \cite{Mihlin}, who stated that a function $\sigma\in C^{\infty}(\mathbb{R}^{n}\setminus \{0\}),$ satisfying estimates of the kind 
\begin{equation}
    |\partial_{\xi}^{\alpha}\sigma(\xi)|\lesssim_{\alpha} |\xi|^{-|\alpha|},\,\,\,\,|\alpha|\leq [n/2]+1,
\end{equation} has a multiplier $A$ (of the Fourier transform\footnote{ Defined for $f\in C^{\infty}_0(\mathbb{R}^n),$ by $\widehat{f}(\xi):=\int\limits_{\mathbb{R}^n}e^{-2\pi i x\cdot \xi}f(x)dx.$ } on $\mathbb{R}^n$) defined by
\begin{equation}
    Af(x)\equiv T_{\sigma}f(x):=\int\limits_{\mathbb{R}^n}e^{2\pi i x\cdot \xi}\sigma(\xi)\widehat{f}(\xi)d\xi,\,\,f\in C^{\infty}_0(\mathbb{R}^n),
\end{equation}admitting  a bounded extension on $L^p(\mathbb{R}^n),$ for $1<p<\infty.$ Subsequent generalisations to Mihlin's theorem were done by  H\"ormander \cite{Hormander1960},  Calder\'on and  Torchinsky in  \cite{CalderonTorchinsky}, Taibleson and Weiss \cite{TaiblesonWeiss}, Baernstein and Sawyer  \cite{BS}, Seeger  \cite{Seeger1,Seeger2,Seeger3} and many others.  We refer  the reader to Grafakos' paper \cite{Grafakos} (and reference therein) for a complete historical  revision and for recent developments about Milhin-H\"ormander and Marcinkiewicz multiplier theorems on $\mathbb{R}^n$.

Extensions of Marcinkiewicz, and H\"ormander-Mihlin criteria have been proved in the context on Lie groups and several spaces of homogeneous type in the context of spectral multipliers of  self-adjoint operators, e.g., sub-Laplacians,  or of other operators with heat kernels  satisfying  Gaussian estimates, with general contexts that go beyond of the objective of this paper. In view of the extensive literature on the field, we will not review it here, but we refer the reader to  \cite{alexo,Anker,CardonaDelgadoRuzhansky2019,Sikora,CowlingSikora,CW3,Stein,NormanWeiss} and to the extensive list of references therein.

In the framework of  Fourier multipliers on compact Lie groups, by using the Calder\'on-Zygmund type theory in  Coifman and De Guzm\'an \cite{CoifmandeGuzman}, the $L^p$-Fourier multipliers  for SU(2) were investigated by Coifman and Weiss in their  classical works \cite{CoifmanWeiss,CW2}. Subelliptic Spectral multipliers for  $L^p(\textnormal{SU}(2))$ were also considered in Cowling and Sikora \cite{CowlingSikora}. Later on, criteria for the  $L^p$-boundedness of Fourier multipliers for  arbitrary compact Lie groups $G$  were given  in  \cite{RuzhanskyWirth2015}, with  a  generalisation in  \cite[Section 5]{SubellipticCalculus} to $L^p$-subelliptic Fourier multipliers.

One of the notable questions when studying the qualitative properties for multipliers on Lie groups is to endow (the spaces of discrete functions on) the unitary dual with a difference structure.  So,  fixing the unitary dual $\widehat{G}$ of  an arbitrary compact Lie group $G,$ when  generalising the Marcinkiewicz condition on a Fourier multiplier $A$ on $G,$ associated to a sequence (called the symbol of $A$) $\sigma:=\{\sigma(\xi)\}_{[\xi]\in \widehat{G}},$\footnote{Here, $[\xi]$ denotes the equivalence class of a unitary, irreducible and continuous representation $\xi:G\rightarrow\textnormal{Hom}(\mathbb{C}^{d_\xi})$ on $G,$  $d_\xi$ is the dimension of its representation space, $\sigma(\xi)\in \textnormal{Hom}(\mathbb{C}^{d_\xi}),$ and $\widehat{f}(\xi):=\int_Gf(x)\xi(x)^{*}dx,$ is the Fourier transform on the group $G,$ of $f\in C^{\infty}(G)$ at $[\xi]$. For instance, in the case of the torus $G=\mathbb{T}^n$, $\xi(x):=e^{i2\pi x\cdot \xi},$ $x\in \mathbb{T}^n,$ $d_\xi\equiv 1,$ and so $\widehat{G}\cong \mathbb{Z}^n.$}
\begin{equation}
    Af(x)\equiv T_\sigma f(x):=\sum_{[\xi]\in \widehat{G}}d_{\xi}\textnormal{Tr}[\xi(x)\sigma(\xi)\widehat{f}(\xi)],\quad f\in C^{\infty}(G),\,
\end{equation}
 the following questions arise:
\begin{itemize}
    \item[(Q1):] how to define the difference operators $\Delta^{\alpha}$ on $\widehat{G},$ in such a way that they generalise the usual notion of difference operators on $\mathbb{Z}\cong \widehat{\mathbb{T}}$?
    \item[(Q2):] which is the required order for the differences operators applied to $\sigma$  in order that $A$ admits a bounded extension on $L^p(G)$?
\end{itemize}
We note that (Q1) was satisfactorily solved  in  \cite{RuzhanskyWirth2015} by introducing a family of difference operators $\Delta^{\alpha}:=\mathbb{D}_{\xi}^{\alpha},$ (defined in terms of the Fourier transform on $G,$ as it was done  in \cite{Ruz})   in terms of the unitary representations $\xi:G\rightarrow \textnormal{End}(\mathbb{C}^{\ell})$\footnote{We will always write $d_\xi:=\ell$ for the dimension of the representation space $\mathbb{C}^{\ell}.$ Also, $I_{\ell}$ is the identity matrix of size $\ell\times \ell.$ } of $G.$ About (Q2), the following Marcinkiewicz type theorem was proved in \cite{RuzhanskyWirth2015}.
\begin{theorem}\label{TRW2015ZIP'}
Let us assume that $G$ is a compact Lie group of dimension $n.$  Let  $\sigma\in \Sigma(\widehat{G})$\footnote{$\Sigma(\widehat{G})$ is the space of functions $\sigma:\widehat{G}\rightarrow \cup_{\ell\in \mathbb{N}_0}(\textnormal{End}(\mathbb{C}^{\ell})).$} be a  symbol satisfying 
\begin{equation}\label{braket}
    \Vert\mathbb{D}^{\alpha} \sigma(\xi)\Vert_{\textnormal{op}}\leqslant C_\alpha\langle \xi\rangle^{-|\alpha|},\,\,|\alpha|\leqslant \varkappa:=\left[\frac{n}{2}\right]+1.
\end{equation} Then $A\equiv T_\sigma$ is of weak type $(1,1)$  and bounded on $L^p(G)$ for all $1<p<\infty.$  Moreover,
\begin{equation}
    \Vert A \Vert_{\mathscr{B}(L^p(G))},\, \Vert A \Vert_{\mathscr{B}(L^1(G),L^{1,\infty}(G))}  \lesssim  \max\{C_{\alpha}:|\alpha|\leq \varkappa\}.
\end{equation}
\end{theorem}
\begin{remark}
In the symbol condition \eqref{braket}, $$\textnormal{Spect}((1+\mathcal{L}_G)^{\frac{1}{2}}):=\{\langle \xi\rangle:[\xi]\in \widehat{G}\},$$ is the system of eigenvalues of the Bessel potential operator $(1+\mathcal{L}_G)^{\frac{1}{2}}$ associated to the Laplacian on $G,$ which can be defined as follows.    Taking an arbitrary orthonormal basis $X_{\mathfrak{g}}:=\{X_{1},\cdots, X_n\}$ of the Lie algebra $\mathfrak{g}$ of $G,$ with respect to the Killing form on $\mathfrak{g},$  $\mathcal{L}_G:=\sum_{i=1}^{n}X_{i}^2.$ We refer the reader to Remark \ref{remarkD} for details about the definition of the difference operators $\mathbb{D}^{\alpha}=\mathbb{D}_1^{\alpha_1}\cdots \mathbb{D}_n^{\alpha_n} .$ They are  compositions of  differences operators  $\mathbb{D}_{j}$ of first order  associated to the  entries of the matrix-function $\xi_0(\cdot)-I_{d_{\xi_0}}$ for any choice of a unitary representation in every equivalence class $[\xi_0]\in\widehat{G}.$
\end{remark}
\begin{remark}
Theorem \ref{TRW2015ZIP'} was proved by using the H\"ormander-Mihlin  Theorem  in \cite[Page 630]{RuzhanskyWirth2015}, see also Theorem \ref{HM}. 
\end{remark}

Let us note that for graded Lie groups (e.g. the Heisenberg group, any stratified group and a wide class of nilpotent Lie groups where Rockland operators  exist, see \cite{FR2} for details) the H\"ormander-Mihlin and the Marcinkiewicz conditions for $L^p,$ Triebel-Lizorkin and Hardy spaces  have been investigated  in  \cite{CardonaRuzhanskyBesovSpaces,CR,CardonaRuzhanskyTriebelGraded,FR}.

In this work we investigate the Marcinkiewicz condition for multipliers on Triebel-Lizorkin spaces $F^{r}_{p,q}(G)$ on $G,$ extending in  Theorem \ref{HMTTL} to the case of compact Lie groups, the estimate of Seeger \cite{Seeger3}  for multipliers in Triebel-Lizorkin spaces $F^{r}_{p,q}(\mathbb{R}^n)$ on $\mathbb{R}^n.$ In order to present our main result, let us define the Triebel-Lizorkin spaces $F^{r}_{p,q}(G),$ as they were introduced by the second author, Nursultanov and Tikhonov in \cite{NRT}. So, let us fix $\eta\in C^{\infty}_0(\mathbb{R}^{+},[0,1]),$ $\eta\neq 0,$ so that $\textnormal{supp}(\eta)\subset [1/2,2],$ and such that 
\begin{equation}
    \sum_{j\in \mathbb{Z}}\eta(2^{-j}\lambda)=1,\,\,\lambda>0.
\end{equation}  Fixing $\psi_0(\lambda):=\sum_{j=-\infty}^{0} \eta_j(\lambda),$ and for $j\geq 1,$ $\psi_j(\lambda):=\eta(2^{-j}\lambda),$ we  have
\begin{equation}
    \sum_{\ell=0}^{\infty}\psi_\ell(\lambda)=1,\,\,\lambda>0,
\end{equation} and one can define the family of operators $\psi_j(\mathcal{B})$ using the functional calculus of the subelliptic Bessel potential $\mathcal{B}:=(1+\mathcal{L}_G)^{\frac{1}{2}}.$ Then, for $0<q<\infty,$ and $1<p<\infty,$ the Triebel-Lizorkin space $F^{r}_{p,q}(G)$ consists of the distributions $f\in \mathscr{D}'(G)$ such that
\begin{equation*}
 \Vert f\Vert_{F^{r}_{p,q}(G)}:=   \left\Vert\left(\sum_{\ell=0}^{\infty}2^{ \ell r q }\left|\psi_{\ell}(\mathcal{B})f\right|^{q}\right)^{\frac{1}{q}} \right\Vert_{L^p(G)}<\infty.
\end{equation*}The weak-$F^{r}_{1,q}(G)$ space  is defined by the distributions $f\in \mathscr{D}'(G)$ such that
\begin{equation}
 \Vert f \Vert_{\textrm{weak-}F^{r}_{1,q}(G)}    :=   \sup_{t>0}t\left|\left\{x\in G:\left(\sum_{\ell=0}^\infty2^{ \ell r q }|\psi_\ell(\mathcal{B}) f(x)|^q   \right)^{\frac{1}{q}}>t \right\}\right|<\infty.
\end{equation}Above, for a measurable subset $A\subset G,$ $|A|$ denotes its Haar measure.
The main results of this work are Theorem \ref{HMTTL} below and the H\"ormander-Mihlin Theorem \ref{CardonaRuzhansky}.
\begin{theorem}\label{HMTTL}
Let us assume that $G$ is a compact Lie group of dimension $n.$  Let  $\sigma\in \Sigma(\widehat{G})$ be a  symbol satisfying 
\begin{equation}\label{braket2}
    \Vert\mathbb{D}^{\alpha} \sigma(\xi)\Vert_{\textnormal{op}}\leqslant C_\alpha\langle \xi\rangle^{-|\alpha|},\,\,|\alpha|\leqslant \varkappa:=\left[\frac{n}{2}\right]+1.
\end{equation}
Then $A\equiv T_\sigma$ extends to a bounded operator from  $F^{r}_{p,q}(G)$ into $F^{r}_{p,q}(G)$  for all $1<p,q<\infty,$ and all $r\in \mathbb{R}.$ For $p=1,$ $A$ admits a bounded extension from $F^{r}_{1,q}(G)$ into $\textrm{weak-}F^{r}_{1,q}(G).$   Moreover
\begin{equation}
\Vert A \Vert_{ \mathscr{B}(F^{r}_{p,q}(G))    },\,
\Vert A \Vert_{\mathscr{B}\left(F^{r}_{1,q}(G),\,\textrm{weak-}F^{r}_{1,q}(G)\right)}    \lesssim  \max\{C_{\alpha}:|\alpha|\leq \varkappa\}.
\end{equation}
\end{theorem}

\begin{remark}
 The Marcinkiewicz Theorem \ref{HMTTL} will be deduced, from the H\"ormander-Mihlin Theorem \ref{CardonaRuzhansky} for Triebel-Lizorkin spaces. In particular, in view of the Littlewood-Paley theorem (see Furioli, Melzi and Veneruso \cite{furioli}), we have $L^p(G)=F^{0}_{p,2}(G)$ for $1<p<\infty,$ so that we recover the $L^p$-bound in Theorem \ref{TRW2015ZIP'}. Because we are not assuming that the Fourier multipliers are defined by the spectral calculus, Theorem \ref{CardonaRuzhansky}  extends the main theorem in Weiss \cite{NormanWeiss} and also the historical 1939's result due to Marcinkiwicz \cite{Marc} (see Remark \ref{finalreamrk} and Corollary \ref{HMTTL22}).  As before, we refer the reader to Remark \ref{remarkD} for details about the definition of the difference operators $\mathbb{D}^{\alpha}=\mathbb{D}_1^{\alpha_1}\cdots \mathbb{D}_n^{\alpha_n} .$ 

\end{remark}

\section{Preliminaries}\label{Preliminaries}

\subsection{The unitary dual  and the Fourier transform}  First, let us record the notion of the unitary dual $\widehat{G}$ of a compact Lie group $G.$  So,  
let us assume that $\xi$ is a continuous, unitary and irreducible  representation of $G,$ this means that,
\begin{itemize}
    \item $\xi\in \textnormal{Hom}(G, \textnormal{U}(H_{\xi})),$ for some finite-dimensional vector space $H_\xi\cong \mathbb{C}^{d_\xi},$ i.e. $\xi(xy)=\xi(x)\xi(y)$ and for the  adjoint of $\xi(x),$ $\xi(x)^*=\xi(x^{-1}),$ for every $x,y\in G.$
    \item The map $(x,v)\mapsto \xi(x)v, $ from $G\times H_\xi$ into $H_\xi$ is continuous.
    \item For every $x\in G,$ and $W_\xi\subset H_\xi,$ if $\xi(x)W_{\xi}\subset W_{\xi},$ then $W_\xi=H_\xi$ or $W_\xi=\emptyset.$
\end{itemize} Let $\textnormal{Rep}(G)$ be the set of unitary, continuous and irreducible representations of $G.$ The relation, {\small{
\begin{equation*}
    \xi_1\sim \xi_2\textnormal{ if and only if, there exists } A\in \textnormal{End}(H_{\xi_1},H_{\xi_2}),\textnormal{ such that }A\xi_{1}(x)A^{-1}=\xi_2(x), 
\end{equation*}}}for every $x\in G,$ is an equivalence relation and the unitary dual of $G,$ denoted by $\widehat{G}$ is defined via
$$
    \widehat{G}:={\textnormal{Rep}(G)}/{\sim}.
$$By a suitable changes of basis, we always can assume that every $\xi$ is matrix-valued and that $H_{\xi}=\mathbb{C}^{d_\xi}.$ If a representation $\xi$ is unitary, then $$\xi(G):=  \{\xi(x):x\in G \}$$ is a subgroup (of the group of  matrices $\mathbb{C}^{d_\xi\times d_\xi}$) which is isomorphic to the original group $G$. Thus the homomorphism $\xi$ allows us to represent the compact Lie group $G$ as a group of matrices. This is the motivation for the term `representation'.
Here, as usually,  
\begin{equation*}
    \widehat{f}(\xi)\equiv (\mathscr{F}f)(\xi):=\int\limits_{G}f(x)\xi(x)^*dx\in  \mathbb{C}^{d_\xi\times d_\xi},\,\,\,[\xi]\in \widehat{G},
\end{equation*}is the matrix-valued Fourier transform of $f\in C^{\infty}_0(G)$ at $\xi=(\xi_{ij})_{i,j=1}^{d_\xi}.$

\subsection{Difference operators} Difference operators on compact Lie groups were introduced in \cite{Ruz} to endow the unitary dual of a compact Lie group with a difference structure. In terms of them, H\"ormander classes of pseudo-differential operators on a compact Lie group can be characterised, see \cite{RuzTurIMRN,RuzhanskyTurunenWirth2014,RuzhanskyWirth2014}. Same as in \cite{RuzhanskyWirth2015}, where differences operators were used to study $L^p$-multipliers we will extend that analysis to the case of Triebel-Lizorkin spaces (see \cite{NurRuzTikhBesov2015,NRT} for instance).

We will denote by  $\Sigma(\widehat{G}) $ the space of  matrix-valued functions,
 \begin{equation*}
    \Sigma(\widehat{G}):=\{ \sigma\in \mathscr{F}(\mathscr{D}'(G))=:\mathscr{D}'(\widehat{G}) \,|\,\sigma:  \widehat{G}\rightarrow \cup_{[\xi]\in \widehat{G}}\mathbb{C}^{d_\xi\times d_\xi}\}.
\end{equation*}By following  \cite{RuzhanskyWirth2015},   a difference operator $Q_\xi$ of order $k,$ can be applied to a symbol $\sigma=\widehat{f}\in \mathscr{D}'(\widehat{G}), $ via
\begin{equation}\label{taylordifferences}
    Q_\xi\sigma(\xi)=\widehat{qf}(\xi),\,[\xi]\in \widehat{G}, 
\end{equation} where $Q_\xi$ is associated with a smooth function $q$ vanishing of order $k$ at the identity $e=e_G.$ We will denote by $\textnormal{diff}^k(\widehat{G})$  the set of all difference operators of order $k.$ For a  fixed smooth function $q,$ the associated difference operator will be denoted by $\Delta_q:=Q_\xi.$ We will choose an admissible collection of difference operators (see e.g. \cite{RuzhanskyWirth2015}),
\begin{equation*}
  \Delta_{\xi}^\alpha:=\Delta_{q_{(1)}}^{\alpha_1}\cdots   \Delta_{q_{(i)}}^{\alpha_{i}},\,\,\alpha=(\alpha_j)_{1\leqslant j\leqslant i}, 
\end{equation*}
where
\begin{equation*}
    \textnormal{rank}\{\nabla q_{(j)}(e):1\leqslant j\leqslant i \}=\textnormal{dim}(G), \textnormal{   and   }\Delta_{q_{(j)}}\in \textnormal{diff}^{1}(\widehat{G}).
\end{equation*}We say that this admissible collection is strongly admissible if 
\begin{equation*}
    \bigcap_{j=1}^{i}\{x\in G: q_{(j)}(x)=0\}=\{e_G\}.
\end{equation*}

\begin{remark}\label{remarkD} A special type of difference operators can be defined by using the unitary representations  of $G.$ Indeed, if $\xi_{0}$ is a fixed irreducible and unitary  representation of $G$, consider the matrix
\begin{equation}
 \xi_{0}(g)-I_{d_{\xi_{0}}}=[\xi_{0}(g)_{ij}-\delta_{ij}]_{i,j=1}^{d_\xi},\, \quad g\in G.   
\end{equation}Then, we associated  to the function 
$
    q_{ij}(g):=\xi_{0}(g)_{ij}-\delta_{ij},\quad g\in G,
$ a difference operator  via
\begin{equation}
    \mathbb{D}_{\xi_0,i,j}:=\mathscr{F}(\xi_{0}(g)_{ij}-\delta_{ij})\mathscr{F}^{-1}.
\end{equation}
If the representation is fixed we omit the index $\xi_0$ so that, from a sequence $\mathbb{D}_1=\mathbb{D}_{\xi_0,j_1,i_1},\cdots, \mathbb{D}_n=\mathbb{D}_{\xi_0,j_n,i_n}$ of operators of this type we define $\mathbb{D}^{\alpha}=\mathbb{D}_{1}^{\alpha_1}\cdots \mathbb{D}^{\alpha_n}_n$, where $\alpha\in\mathbb{N}^n$.
\end{remark}
\begin{remark}[Leibniz rule for difference operators]\label{Leibnizrule} The difference structure on the unitary dual $\widehat{G},$ induced by the difference operators acting on the momentum variable $[\xi]\in \widehat{G},$  implies the following Leibniz rule 
\begin{align*}
    \Delta_{q_\ell}(a_{1}a_{2})(x_0,\xi) =\sum_{ |\gamma|,|\varepsilon|\leqslant \ell\leqslant |\gamma|+|\varepsilon| }C_{\varepsilon,\gamma}(\Delta_{q_\gamma}a_{1})(x_0,\xi) (\Delta_{q_\varepsilon}a_{2})(x_0,\xi), \quad (x_{0},[\xi])\in G\times \widehat{G},
\end{align*} for $a_{1},a_{2}\in C^{\infty}(G, \mathscr{S}'(\widehat{G})).$ For details we refer the reader to  \cite{Ruz}.
\end{remark}
\begin{remark}[Difference operators of fractional order and Sobolev spaces on the unitary dual] In the spirit of the Sobolev spaces on the unitary dual of a graded Lie group \cite{FischerRuzhansky2017}, Sobolev spaces also can be defined for the unitary dual of a compact Lie group. They can be defined as follows:
Let $\Delta_{q_{1}},$ be a difference operator of first order associated to a smooth and non-negative function $q_1\geq 0.$ For $s\in \mathbb{R},$ the Sobolev space $\dot{L}^{2}_s(\widehat{G})$ consists of all distributions $\sigma=\widehat{f}\in \mathscr{D}'(G)$ such that
 \begin{equation}
      \Vert\sigma \Vert_{\dot{L}^{2}_s(\widehat{G})}:=\Vert f\Vert_{L^2(G,{q_{1}^{2s}})}=\Vert {q_{1}^{s}} f\Vert_{L^2(G)}<\infty.
 \end{equation}For $s\in \mathbb{N},$ observe that,
 \begin{equation}
     \Vert\sigma \Vert_{\dot{L}^{2}_s(\widehat{G})}\asymp\max_{|\alpha|=s}\Vert\Delta_{\xi}^{\alpha} \sigma\Vert_{L^2(\widehat{G})}\asymp \max_{|\alpha|=s}\Vert\mathbb{D}_{\xi}^{\alpha} \sigma\Vert_{L^2(\widehat{G})}.
 \end{equation}
 So, for every $s\in \mathbb{R},$ the difference operator $ \Delta_{q_{1}}^{s}:=\Delta_{q_{1}^{s}}$ of fractional order $s,$ can be defined in terms of the Fourier transform, via:
 \begin{equation}\label{eqref}
     \Delta_{q_{1}}^{s}\widehat{f}=\widehat{ q_{1}^{s} f},\quad \,f\in \mathscr{D}'(G).
 \end{equation} So, we have $\Vert\sigma \Vert_{\dot{L}^{2}_s(\widehat{G})}:=\Vert\Delta_{q_{1}}^{s}\widehat{f}\Vert_{L^2(\widehat{G})}.$ We will denote
 \begin{equation}
     \textnormal{diff}^{\,s}(\widehat{G}):=\{\Delta_{q_{1}^{s}}: \Delta_{q_{1}}\in  \textnormal{diff}^{1}(\widehat{G}) \}.
 \end{equation}
\end{remark}

\subsection{Calder\'on-Zygmund type estimates for multipliers}
In order to provide $L^p$-estimates for multipliers in the subelliptic context, we will use the techniques developed by the second author and J. Wirth in \cite{RuzhanskyWirth2015}, where a special case (compatible with the notion of difference operators and the difference structure that they provide for the unitary dual) of a statement of Coifman and de Guzm\'an (\cite{CoifmandeGuzman}, Theorem 2) was established.    We record it as follows  (see   \cite[p. 630]{RuzhanskyWirth2015}).  
    \begin{theorem}[H\"ormander-Mihlin  Theorem  for $L^p(G)$]\label{CoifDeGuzCrit} Assume that $A:L^2(G)\rightarrow L^2(G)$ is a left-invariant operator on $G$ satisfying 
 \begin{equation}\label{CoifDeGuz}
  \Vert A\psi_{r^{-1}}\Vert_{L^2(G,\rho(x)^{n(1+\varepsilon)}dx )}  := \left(\,\int\limits_{G}|A\psi_{{r^{-1}}}(x)|^2\rho(x)^{n(1+\varepsilon)}dx\right)^\frac{1}{2}\lesssim  Cr^{-\frac{\varepsilon}{2}},
 \end{equation}for some  $\varepsilon>0,$ uniformly in $r>0.$ Then $A$ is of weak type $(1,1)$ and bounded on $L^p(G),$ for all $1<p<\infty.$
 \end{theorem}
 The family $\{\psi_r\}_{r>0}$  that appears in  Theorem  \ref{CoifDeGuzCrit} is defined by  \begin{equation}\label{continuos resolution}
 \psi_{r}:=\phi_{r}-\phi_{r/2},
 \end{equation}
 where the functions in the net $\{\phi_r\}_{r>0},$ satisfy, among other things,  the following properties (see   \cite[Lemma 3.3]{RuzhanskyWirth2015}):
 \begin{itemize}
     \item $\int\limits_{G}\phi_{r}(x)dx=1,$ 
     \item $\int\limits_{G}\phi_{r}^2(x)dx=O(\frac{1}{r}),$
     \item $\phi_r\ast\phi_s=\phi_s\ast \phi_r,$ $r,s>0.$
 \end{itemize}
  The function $\rho:x\mapsto\rho(x),$  appearing in \eqref{CoifDeGuz}, is a suitable pseudo-distance defined on $G.$ If $G$ is semi-simple (this means that the centre of $G,$ $Z(G)$ is trivial), it is defined by
  \begin{equation}\label{therhofucntion}
      \rho(x)^2:=\dim(G)-\textnormal{\textbf{Tr}}(\textnormal{Ad}(x))=\sum_{\xi\in \Delta_0}(d_\xi-\textnormal{\textbf{Tr}}(\xi(x))),\,\,x\in G,
  \end{equation}
 where $\textnormal{Ad}:G\rightarrow\textnormal{U}(\mathfrak{g}),$ and $\Delta_0$ is the system of positive roots. It can be decomposed into irreducible representations as,
 \begin{equation*}
     \textnormal{Ad}=[\textnormal{rank}(G)e_{\widehat{G}}]\oplus\left( \bigoplus_{\xi\in \Delta_{0}}\xi\right),
 \end{equation*}
 where $e_{\widehat{G}}$ is the trivial representation.
With the consideration on the centre $Z(G)=\{e_{G}\},$ it can be shown (see Lemma 3.1 of \cite{RuzhanskyWirth2015}) that 
\begin{itemize}
    \item $\rho^2(x)\geqslant   0$ and $\rho(x)=0$ if and only if $x=e_{G}.$
    \item $\Delta_{\rho^2}\in \textnormal{diff}^2(\widehat{G}).$
\end{itemize}
 If $G$ is not semi-simple,  we refer the reader to \cite[Remark 3.2]{RuzhanskyWirth2015} for the modifications in the definition of $\rho,$ in this particular case. The construction of the functions $\phi_r,$ is as follows. By choosing $\tilde{\phi}\in C^{\infty}_0(\mathbb{R})$ such that $\tilde{\phi}\geq 0,$ $\tilde{\phi}(0)\equiv 1,$ and $\partial_{t}^{\ell}\tilde{\phi}(0)=0,$ for $\ell \geq 1,$ we define
 \begin{equation}
     \psi_{r}(g):=c_{r}\tilde{\psi}(r^{-\frac{1}{n}}\rho(g)),\,\,\quad \int\limits_{G}\psi_{r}(g)dg=1,
 \end{equation}with the normalisation condition used to define $c_{r}.$ An equivalent statement to the H\"ormander-Mihlin  Theorem  \ref{CoifDeGuzCrit} will be given in terms of the family of functions
$
     \varphi_{r^{-1}}\equiv \psi_{r^{-n}},\,r>0,
 $ (see Theorem \ref{HM}) that provide a continuous dyadic decomposition.

\subsection{$L^p(G)$-boundedness of Fourier multipliers}\label{Lpcompact1} In this subsection we recall the H\"ormander-Mihlin theorem and the Marcinkiewicz-Mihlin theorem for $L^p$-multipliers on compact Lie groups  \cite{RuzhanskyWirth2015}.

\begin{remark}Let $A$ be a Fourier multiplier with symbol $\sigma.$
 Let us observe that that   Theorem  \ref{CoifDeGuzCrit} can be re-written in terms of the Sobolev spaces $\dot{L}^{2}_s(\widehat{G})$ on the unitary dual, showing that,  \eqref{CoifDeGuz} is indeed, an analogue on compact Lie groups of the H\"ormander-Mihlin theorem in \cite{Hormander1960}.  So, with the notation in  Theorem  \ref{CoifDeGuzCrit}, and making use of the Plancherel theorem, we have 
 \begin{align*}
   \Vert A\psi_{r}\Vert_{L^2(G,\rho(x)^{n(1+\varepsilon)dx})}^2
   &:= \,\int\limits_{G}|A\psi_{r}(x)|^2\rho(x)^{n(1+\varepsilon)}dx
   \\
   &=\Vert\Delta_{\rho(x)}^{{n(1+\varepsilon)}/{2}} \sigma(\xi)\widehat{\psi}_r(\xi)\Vert^2_{L^{2}(\widehat{G})}\\
   &:=\Vert \sigma(\xi)\widehat{\psi}_r(\xi)  \Vert^2_{L^{2}_{s}(\widehat{G})},
 \end{align*}
where $s:={n(1+\varepsilon)}/{2}>n/2.$ Observing that $\frac{\varepsilon}{2}=\frac{1}{n}(s-\frac{n}{2}),$ and that  $A:L^{2}(G)\rightarrow L^2(G)$ is  bounded if and only if, $  \Vert\sigma\Vert_{L^{\infty}(\widehat{G})}:=\sup_{[\xi]\in \widehat{G}}\Vert \sigma(\xi)\Vert_{\textnormal{op}}<\infty,$    Theorem  \ref{CoifDeGuzCrit} is equivalent to the following theorem. We set $$\varphi_{r^{-1}}\equiv \psi_{r^{-n}}:=\phi_{r^{-n}}-\phi_{r^{-n}/2},\,\quad r>0,$$  for the continuous dyadic approximation of the identity in \eqref{continuos resolution}.
 \end{remark}
 \begin{theorem}[H\"ormander-Mihlin Theorem for $L^p(G)$]\label{HM} Let $G$ be a compact Lie group   and let $s>\frac{n}{2}$. Let $\sigma\in \Sigma(\widehat{G})$ be a symbol satisfying
\begin{equation}\label{RuzhanskyWirth2015hypo}
  \Vert\sigma \Vert_{l.u.,\, L_{s}^2(\widehat{G})}:=  \Vert\sigma\Vert_{L^{\infty}(\widehat{G})}+ \sup_{r>0}r^{(s-\frac{n}{2})}\Vert \sigma\cdot \widehat{\varphi}_{r^{-1}}\Vert_{\dot{L}^2_s(\widehat{G})} <\infty.
\end{equation} Then $A\equiv T_\sigma$ is of weak type $(1,1)$  and bounded on $L^p(G)$ for all $1<p<\infty.$
  
 \end{theorem}
We present Theorem \ref{TRW2015ZIP'} in the following form.

\begin{theorem}[Marcinkiewicz-Mihlin Theorem]\label{TRW2015ZIP}
Let $G$ be a compact Lie group   and let $\varkappa\in 2\mathbb{N}$ be such that $\varkappa>\frac{n}{2}$. Let $\sigma\in \Sigma(\widehat{G})$ be a symbol satisfying
\begin{equation*}
    \Vert\mathbb{D}^{\alpha} \sigma(\xi)\Vert_{\textnormal{op}}\leqslant C_\alpha\langle \xi\rangle^{-|\alpha|},\,\,|\alpha|\leqslant \varkappa.
\end{equation*} Then $A\equiv T_\sigma$ is of weak type $(1,1)$  and bounded on $L^p(G)$ for all $1<p<\infty.$
\end{theorem}
\begin{remark}
Several properties for the Sobolev spaces $L^2_s(\widehat{G})$ on the unitary dual were established in \cite{FischerDiff}. In particular, the H\"ormander-Mihlin Theorem \ref{HM} was re-obtained, by observing that the $L^p(G)$-boundedness and the weak (1,1) type of a Fourier multiplier, also can be obtained if one verifies the following condition:
\begin{equation}\label{RuzhanskyWirth2015hypo2}
  \Vert\sigma \Vert_{l.u.,\, L_{s}^2(\widehat{G})}':=  \Vert\sigma\Vert_{L^{\infty}(\widehat{G})}+ \sup_{r>0}r^{(s-\frac{n}{2})}\Vert \sigma\cdot \eta({r^{-1}\langle \xi\rangle})\Vert_{\dot{L}^2_s(\widehat{G})} <\infty,
\end{equation} for any $\eta\neq 0,$ $\eta\in C^{\infty}_0(\mathbb{R}^{+}_0),$ and $s>n/2.$ Both H\"ormander-Mihlin conditions (\eqref{RuzhanskyWirth2015hypo} or \eqref{RuzhanskyWirth2015hypo2}) allow   us to write Theorem \ref{TRW2015ZIP} in the form of  Theorem \ref{TRW2015ZIP'}.
\end{remark}
\begin{remark}
The  boundedness of pseudo-differential operators on compact Lie groups, including  oscillating Fourier multipliers, on $L^p,$ subelliptic Sobolev and Besov spaces  can be found in \cite{CardonaRuzhanskyTriebelGraded,RuzhanskyDelgado2017} and \cite{Cardona2,Cardona3,CR}. The boundedness of Fourier multipliers on $L^p$-spaces, Triebel-Lizorkin spaces and Hardy spaces on graded Lie groups can be found in \cite{FR,CardonaRuzhanskyTriebelGraded} and \cite{HongHuRuzhansky2020}, respectively.
\end{remark}

\subsection{Triebel-Lizorkin spaces on compact Lie groups} In this section, Triebel-Lizorkin spaces  on compact Lie groups are introduced. This family of spaces was introduced  by the second author, Nursultanov and Tikhonov in \cite{NRT} by using a dyadic partition of the spectral resolution of the Bessel potential $\mathcal{B}=(1+\mathcal{L}_G)^{\frac{1}{2}}$.

As in the introduction, let us fix $\eta\in C^{\infty}_0(\mathbb{R}^{+},[0,1]),$ $\eta\neq 0,$ so that $\textnormal{supp}(\eta)\subset [1/2,2],$ and such that 
\begin{equation}
    \sum_{j\in \mathbb{Z}}\eta(2^{-j}\lambda)=1,\,\,\lambda>0.
\end{equation}  Fixing $\psi_0(\lambda):=\sum_{j=-\infty}^{0} \eta_j(\lambda),$ and for $j\geq 1,$ $\psi_j(\lambda):=\eta(2^{-j}\lambda),$ we  have
\begin{equation}
    \sum_{\ell=0}^{\infty}\psi_\ell(\lambda)=1,\,\,\lambda>0.
\end{equation} Let us  define the family of operators $\psi_j(\mathcal{B})$ using the functional calculus. Then, for $0<q<\infty,$ and $1<p<\infty,$ the Triebel-Lizorkin space $F^{r }_{p,q}(G)$ consists of the distributions $f\in \mathscr{D}'(G)$ such that
\begin{equation*}
 \Vert f\Vert_{F^{r }_{p,q}(G)}:=   \left\Vert\left(\sum_{\ell=0}^{\infty}2^{ \ell r q }\left|\psi_{\ell}(\mathcal{B})f\right|^{q}\right)^{\frac{1}{q}} \right\Vert_{L^p(G)}<\infty,
\end{equation*}and for $p=1,$ the weak-$F^{r}_{1,q}(G)$ space  is defined by the distributions $f\in \mathscr{D}'(G)$ such that
\begin{equation}
 \Vert f \Vert_{\textrm{weak-}F^{r }_{1,q}(G)}    :=   \sup_{t>0}t\left|\left\{x\in G:\left(\sum_{\ell=0}^\infty2^{ \ell r q }|\psi_\ell(\mathcal{B}) f(x)|^q   \right)^{\frac{1}{q}}>t \right\}\right|<\infty.
\end{equation}

In the following theorem we present some embedding properties for Triebel-Lizorkin spaces. For the proof of Theorem \ref{independence} and for a consistent investigation of Triebel-Lizorkin spaces on compact Lie groups, we refer the reader to \cite{NRT}.

\begin{theorem}\label{independence}
 Let $G$ be a compact  Lie group.  Then we have the following properties:
\begin{itemize}
    \item[(1)] $F^{r+\varepsilon,\mathcal{L} }_{p,q_1}(G)\hookrightarrow F^{r,\mathcal{L} }_{p,q_1}(G) \hookrightarrow F^{r,\mathcal{L} }_{p,q_2}(G) \hookrightarrow F^{r,\mathcal{L} }_{p,\infty}(G),$ $\varepsilon >0,$ $0\leq p\leq \infty,$ $0\leq q_1\leq q_2\leq \infty.$
    \item[(2)] $F^{r+\varepsilon,\mathcal{L} }_{p,q_1}(G) \hookrightarrow F^{r,\mathcal{L} }_{p,q_2}(G), $ $\varepsilon >0,$ $0\leq p\leq \infty,$ $1\leq q_2< q_1< \infty.$
    \item[(3)] $F^{r}_{p,2}(G)=L^p_r(G)$ for all $r\in \mathbb{R},$ and all $1<p<\infty,$ where $L^p_r(G)$ are the standard Sobolev spaces on $G.$
\end{itemize}
 \end{theorem}
\begin{remark}
Let us observe that (3) in Theorem \ref{independence} is a consequence of the Littlewood-Paley theorem. For details we refer the reader to Furioli, Melzi and Veneruso \cite{furioli}.
\end{remark}

\section{H\"ormander-Mihlin multipliers on Triebel-Lizorkin spaces: Proof of Theorem \ref{HMTTL}}

Let us fix $\eta\in C^{\infty}_0(\mathbb{R}^{+},[0,1]),$ $\eta\neq 0,$ so that $\textnormal{supp}(\eta)\subset [1/2,2],$ and such that 
\begin{equation}
    \sum_{j\in \mathbb{Z}}\eta(2^{-j}\lambda)=1,\,\,\lambda>0.
\end{equation}
By defining $\psi_0(\lambda):=\sum_{j=-\infty}^{0} \eta_j(\lambda),$ and for $j\geq 1,$ $\psi_j(\lambda):=\eta(2^{-j}\lambda),$ we obviously have
\begin{equation}
    \sum_{\ell=0}^{\infty}\psi_\ell(\lambda)=1,\,\,\lambda>0.
\end{equation}

First, let us deduce the proof of Theorem \ref{HMTTL} from the following result. The distribution  $\vartheta_{t}:=\eta(t\mathcal{B})\delta$ denotes the right-convolution kernel of the operator $\eta(t\mathcal{B}).$ 
\begin{theorem}[H\"ormander-Mihlin  Theorem  for $F^{r}_{p,q}(G)$]\label{CardonaRuzhansky}  Let us consider a Fourier multiplier $A:C^{\infty}(G)\rightarrow L^2(G)$ satisfying the symbol condition
\begin{equation}\label{toverifyHM}
  \Vert\sigma \Vert_{l.u.,\, L_{s}^2(\widehat{G})}':=  \Vert\sigma\Vert_{L^{\infty}(\widehat{G})}+ \sup_{r>0}r^{(s-\frac{n}{2})}\Vert \sigma\cdot \eta({r^{-1}\langle \xi\rangle})\Vert_{\dot{L}^2_s(\widehat{G})} <\infty,
\end{equation} with $s>\frac{n}{2}.$
 Then $A\equiv T_\sigma$ extends to a bounded operator from  $F^{r}_{p,q}(G)$ into $F^{r}_{p,q}(G)$  for all $1<p,q<\infty,$ and all $r\in \mathbb{R}.$ For $p=1,$ $A$ admits a bounded extension from $F^{r}_{1,q}(G)$ into $\textrm{weak-}F^{r}_{1,q}(G).$   Moreover
\begin{equation}\label{equationnorms}
\Vert A \Vert_{ \mathscr{B}(F^{r}_{p,q}(G))    },\,
\Vert A \Vert_{\mathscr{B}\left(F^{r}_{1,q}(G),\,\textrm{weak-}F^{r}_{1,q}(G)\right)}    \lesssim  \Vert\sigma \Vert_{l.u.,\, L_{s}^2(\widehat{G})}' .
\end{equation}
 \end{theorem}
\begin{remark} Note that the Fourier transform of the distribution  $\vartheta_{t}:=\eta(t\mathcal{B})\delta$ is given by 
\begin{equation}\label{nut:transform}
\widehat{\vartheta}_{t}(\xi):=\eta(t\langle \xi\rangle)I_{d_\xi}, \,\,[\xi]\in \widehat{G},    
\end{equation} and that $A:L^2(G)\rightarrow L^2(G)$ is bounded  if and only if $\Vert\sigma\Vert_{L^\infty(\widehat{G})}<\infty.$    Let $\Delta_{q_1^s}\in \textnormal{diff}^{\,s}(\widehat{G})$ be a fractional difference operator of order $s:=\frac{n+\varepsilon}{2},$ with $\varepsilon>0.$ If $A:L^2(G)\rightarrow L^2(G)$ is bounded and it satisfies the (Coifman-Weiss type)  condition \begin{equation}\label{CardonaRuzhansky2}
  \Vert A\vartheta_{r^{-1}}\Vert_{L^2(G,q_{1}(x)^{\varepsilon+n}dx )}  := \left(\,\int\limits_{G}| A\vartheta_{r^{-1}}(x)|^2q_{1}(x)^{\varepsilon+n}dx\right)^\frac{1}{2}\lesssim_{\varepsilon} Cr^{-\frac{1}{2}\varepsilon},
 \end{equation} uniformly in $r>0,$
the Plancherel Theorem makes   the hypothesis in Theorem \ref{CardonaRuzhansky}  equivalent to \eqref{CardonaRuzhansky2}. 
 Moreover, \eqref{equationnorms} is equivalent to
\begin{equation}
 \Vert A \Vert_{ \mathscr{B}(F^{r}_{p,q}(G))    },\,
\Vert A \Vert_{\mathscr{B}\left(F^{r}_{1,q}(G),\,\textrm{weak-}F^{r}_{1,q}(G)\right)}    \lesssim \Vert \sigma\Vert_{L^{\infty}(\widehat{G})}+\sup_{r>0}r^{s-\frac{n}{2}} \Vert A\vartheta_{r^{-1}}\Vert_{L^2(G,q_{1}(x)^{2s}dx )}, 
\end{equation}in view of \eqref{nut:transform}.
\end{remark}
Four our further analysis we will use the following auxiliary estimate.
\begin{lemma}\label{lemmaofauxiliaruse}
Let $A$ be a Fourier multiplier satisfying the hypothesis in Theorem \ref{HMTTL}. Then, the following estimate holds,
\begin{equation}\label{HMconditionGcompact}
\Vert \sigma(\xi)\psi_{\ell}(\langle \xi\rangle)\Vert_{\dot{L}^s_2(\widehat{G})}\lesssim_{s,\psi} 2^{-\ell(s-\frac{n}{2})}\sup_{|\alpha|\leq s}\sup_{[\xi]\in \widehat{G}}\Vert(\Delta^{\alpha}_\xi\sigma(\xi))\langle\xi\rangle^{|\alpha|} \Vert_{\textnormal{op}},\,\,\,\ell\in \mathbb{Z},
\end{equation}provided that $s>n/2,$ $s\in \mathbb{N}.$ Moreover, the right-convolution kernel $\kappa_{\ell}$ of $A\psi_{\ell},$ satisfies the uniform estimate in $\ell\in \mathbb{Z},$
\begin{equation}\label{3.8}
    \int\limits_{ |x|>4c|z|}|\kappa_{\ell}( z^{-1}x)  - \kappa_{\ell}(x)|dx\lesssim 2^{-\ell(s-\frac{n}{2})}\Vert\sigma \Vert_{l.u.,\, L_{s}^2(\widehat{G})}'.
\end{equation}
\end{lemma}
\begin{proof}That \eqref{3.8} is a consequence of \eqref{HMconditionGcompact} was proved in \cite[Page 38]{FischerDiff}. So, let us continue with the first part of the lemma. Let us observe that, (in view of Proposition 4.13 in \cite{FischerDiff}), we have the following norm estimates
\begin{equation}\label{estimateofnormasraro}
    \Vert \sigma \tau\Vert_{\dot{L}^2_{s}(\widehat{G})}\lesssim_s\sum_{s_1+s_2=s}  \Vert\sigma\Vert_{\dot{L}^{\infty}_{s_1}(\widehat{G})}\Vert \tau\Vert_{\dot{L}^{2}_{s_2}(\widehat{G})},\,\quad
    \Vert\sigma\Vert_{\dot{L}^{\infty}_{s_1}(\widehat{G})}:=\max_{|\alpha|=s_1}\Vert \Delta^{\alpha}_\xi\sigma\Vert_{L^{\infty}(\widehat{G})},
\end{equation} in the sense that if the right-hand side is finite then the left-hand side is finite and the inequality holds.  So, applying \eqref{estimateofnormasraro} to $\tau=\psi_\ell(\langle\xi\rangle)I_{d_\xi},$ we have
\begin{align*}
    \Vert \sigma \psi_\ell(\langle\xi\rangle)\Vert_{\dot{L}^2_{s}(\widehat{G})}  &\lesssim_s\sum_{s_1+s_2=s}  \max_{|\alpha|=s_1}\Vert \Delta^{\alpha}_\xi\sigma\Vert_{L^{\infty}(\widehat{G})}\Vert \psi_\ell(\langle\xi\rangle)I_{d_\xi}\Vert_{\dot{L}^{2}_{s_2}(\widehat{G})}.
\end{align*}Observe that,
\begin{align*}
     \Vert \Delta^{\alpha}_\xi\sigma\Vert_{L^{\infty}(\widehat{G})}\leq \Vert \Delta^{\alpha}_\xi\sigma\langle\xi\rangle^{|\alpha|}\Vert_{L^{\infty}(\widehat{G})}\Vert \widehat{\mathcal{B}}(\xi)^{-|\alpha|}\Vert_{L^{\infty}(\widehat{G})}\leq \Vert \Delta^{\alpha}_\xi\sigma\langle\xi\rangle^{|\alpha|}\Vert_{L^{\infty}(\widehat{G})}.
\end{align*}Also, in view of Lemma 5.4 of \cite{FischerDiff},  for any positive and  smooth function $f\in C^{\infty}_{0}(\mathbb{R}),$ supported in $[0,a]$ with $a>0,$ one has 
\begin{equation}
    \Vert f(t\langle\xi\rangle)I_{d_\xi}\Vert_{{\dot{L}^2_{s}(\widehat{G})}} \lesssim_{s}t^{s-\frac{n}{2}}\Vert f \Vert_{H^{s'}(\mathbb{R})},\,s'>s+1/2,\,0<t\leq 1,
\end{equation}which applied to $f=\psi$ and $t=2^{-\ell},$ gives
\begin{equation*}
    \Vert \psi_\ell(\langle\xi\rangle)I_{d_\xi}\Vert_{\dot{L}^{2}_{s_2}(\widehat{G})}\lesssim \Vert \psi_\ell(\langle\xi\rangle)I_{d_\xi}\Vert_{\dot{L}^{2}_{s}(\widehat{G})}\lesssim_{s}2^{-\ell(s-\frac{n}{2})}\Vert \psi \Vert_{H^{s'}(\mathbb{R})}.
\end{equation*}The analysis above shows that 
\begin{align*}
    \Vert \sigma(\xi)\psi_{\ell}(\langle \xi\rangle)\Vert_{\dot{L}^s_2(\widehat{G})}\lesssim_{s}\sup_{|\alpha|\leq s}\Vert \Delta^{\alpha}_\xi\sigma\langle\xi\rangle^{|\alpha|}\Vert_{L^{\infty}(\widehat{G})}2^{-\ell(s-\frac{n}{2})}.
\end{align*}Thus, the proof is complete.
\end{proof}

\begin{proof}[Proof of Theorem \ref{HMTTL}] In view of Lemma \ref{lemmaofauxiliaruse}, we have that
\begin{align*}
  \Vert\sigma \Vert_{l.u.,\, L_{s}^2(\widehat{G})}'&:=  \Vert\sigma\Vert_{L^{\infty}(\widehat{G})}+ \sup_{r>0}r^{(s-\frac{n}{2})}\Vert \sigma\cdot \eta({r^{-1}\langle \xi\rangle})\Vert_{\dot{L}^2_s(\widehat{G})}\\
  &\lesssim_{s,\psi} \sup_{|\alpha|\leq s}\sup_{[\xi]\in \widehat{G}}\Vert(\Delta^{\alpha}_\xi\sigma(\xi))\langle\xi\rangle^{|\alpha|}\Vert_{\textnormal{op}}\lesssim   \sup_{|\alpha|\leq s} C_{\alpha}.
\end{align*}In view of the H\"ormander-Mihlin  Theorem \ref{CardonaRuzhansky} applied to $s:=[\frac{n}{2}]+1$,    $A\equiv T_\sigma$ extends to a bounded operator from  $F^{r}_{p,q}(G)$ into $F^{r}_{p,q}(G)$  for all $1<p,q<\infty,$ and all $r\in \mathbb{R}.$ For $p=1,$ $A$ admits a bounded extension from $F^{r}_{1,q}(G)$ into $\textrm{weak-}F^{r}_{1,q}(G).$ 
\end{proof}

\begin{proof}[Proof of Theorem \ref{CardonaRuzhansky}] 
By observing that  $(1+\mathcal{B})^{\frac{r}{2}}:F^{r}_{p,q}(G)\rightarrow F^{0}_{p,q}(G)$ and $ (1+\mathcal{B})^{-\frac{r}{2}}: F^{0}_{p,q}(G)\rightarrow F^{r}_{p,q}(G)$ are both isomorphisms,  it is enough to prove that $A$ admits a bounded extension from $ F^{0}_{p,q}(G)$ into $ F^{0}_{p,q}(G).$ For this, let us define the vector-valued operator $W:L^2(G,\mathcal{\ell}^2(\mathbb{N}_0))\rightarrow L^2(G,\mathcal{\ell}^2(\mathbb{N}_0))$ by
\begin{equation}
    W(\{g_{\ell}\}_{\ell=0}^{\infty}):=  (\{W_\ell g_{\ell}\}_{\ell=0}^{\infty}),\,\,W_{\ell}:=A\psi_{\ell}(\mathcal{B}).
\end{equation} Observe that $W$ is well-defined (bounded from $L^2(G,\mathcal{\ell}^2(\mathbb{N}_0))$ into  $L^2(G,\mathcal{\ell}^2(\mathbb{N}_0)$)), because $A$ admits a bounded extension on $L^2(G)$ and also, in view of the following estimate
\begin{align*}
  \Vert W(\{g_{\ell}\}_{\ell=0}^{\infty})\Vert^2_{ L^2(G,\mathcal{\ell}^2(\mathbb{N}_0))}
  &:=  \int\limits_{G} \sum_{\ell=0}^\infty |A\psi_\ell(\mathcal{B}) g_\ell(x)|^2dx=\sum_{\ell=0}^\infty\int\limits_{G}|A\psi_\ell(\mathcal{B}) g_\ell(x)|^2dx\\
    &\leq \Vert A \Vert^2_{\mathscr{B}(L^2(G))}\sup_{\ell}\Vert \psi_\ell(\mathcal{B}) \Vert^2_{\mathscr{B}(L^2(G))}\sum_{\ell=0}^\infty\int\limits_{G}  | g_\ell(x)|^2dx\\
    &\leq \Vert A \Vert^2_{\mathscr{B}(L^2(G))}\sup_{\ell}\Vert \psi_\ell \Vert^2_{L^{\infty}(\mathbb{R})}\sum_{\ell=0}^\infty\int\limits_{G}  | g_\ell(x)|^2dx\\
    &= \Vert A \Vert^2_{\mathscr{B}(L^2(G))}\Vert \psi \Vert^2_{L^{\infty}(\mathbb{R})}\sum_{\ell=0}^\infty\int\limits_{G}  | g_\ell(x)|^2dx\\
    &\lesssim  \Vert \{g_{\ell}\}_{\ell=0}^{\infty}\Vert^2_{ L^2(G,\mathcal{\ell}^2(\mathbb{N}_0))}. 
\end{align*}
So, observe that, in order to prove Theorem \ref{CardonaRuzhansky} it is enough to prove the following two lemmas and the estimate \eqref{equationnorms}.
\begin{lemma}\label{Lemma1} Let $G$ be a compact Lie group and let $1<q<\infty.$ Then,
$W:L^q(G,\mathcal{\ell}^q(\mathbb{N}_0))\rightarrow L^q(G,\mathcal{\ell}^q(\mathbb{N}_0))$ admits a bounded extension.
\end{lemma}
\begin{lemma}\label{Lemma2} Let $G$ be a compact Lie group and let $1<q<\infty.$ Then,
$W:L^1(G,\mathcal{\ell}^q(\mathbb{N}_0))\rightarrow L^{1,\infty}(G,\mathcal{\ell}^q(\mathbb{N}_0))$ admits a bounded extension.
\end{lemma}
\begin{corollary} Let $G$ be a compact Lie group and let $1<p,q<\infty.$ Then,  $W:L^p(G,\mathcal{\ell}^q(\mathbb{N}_0))\rightarrow L^{p}(G,\mathcal{\ell}^q(\mathbb{N}_0))$ admits a bounded extension.
\end{corollary}
\begin{proof}
Indeed, by the Marcinkiewicz interpolation, these two lemmas are enough to show that $W:L^p(G,\mathcal{\ell}^q(\mathbb{N}_0))\rightarrow L^{p}(G,\mathcal{\ell}^q(\mathbb{N}_0))$ admits a bounded extension for all $1<p\leq q<\infty.$ The case $1<q\leq p<\infty$ follows from the fact that  $L^{p'}(G,\mathcal{\ell}^{q'}(\mathbb{N}_0))$ is the dual of $L^{p}(G,\mathcal{\ell}^{q}(\mathbb{N}_0))$ and also that  Lemma \ref{Lemma1} and  Lemma \ref{Lemma2} hold if we change $A$ by its standard $L^2$-adjoint.
\end{proof}
\begin{remark}[Lemmas \ref{Lemma1} and \ref{Lemma2} imply Theorem \ref{CardonaRuzhansky}]\label{RemarkInterpolation} By defining $\psi_{-1}=\psi_0,$ we have
\begin{align*}
  \Vert Af\Vert_{F^{0}_{p,q}(G)}&=  \left\Vert \left(\sum_{l=0}^{\infty} |A\psi_{l}(\mathcal{B})f|^{q}    \right)^{\frac{1}{q}}\right\Vert_{L^{p}(G)}\\
  &\leq    \left\Vert \left(\sum_{l=0}^{\infty} |A\psi_{l}(\mathcal{B})[\psi_{l-1}(\mathcal{B})+\psi_{l}(\mathcal{B})+\psi_{l+1}(\mathcal{B})] f|^{q}    \right)^{\frac{1}{q}}\right\Vert_{L^{p}(G)}\\
  &=\Vert W(\{[\psi_{l-1}(\mathcal{B})+\psi_{l}(\mathcal{B})+\psi_{l+1}(\mathcal{B})] f\}_{l=0}^{\infty}) \Vert_{L^p(\ell^q)}\\
  &\lesssim \Vert \{ [\psi_{l-1}(\mathcal{B})+\psi_{l}(\mathcal{B})+\psi_{l+1}(\mathcal{B})] f\}_{l=0}^{\infty}) \Vert_{L^p(\ell^q)} \\
  &\lesssim \Vert f \Vert_{F^{0}_{p,q}(G)}. 
\end{align*} Also note that from Lemma \ref{Lemma2}, we have

\begin{align*}
  \Vert Af\Vert_{\textrm{weak-}F^{0}_{1,q}(G)}&=  \left\Vert \left(\sum_{l=0}^{\infty} |A\psi_{l}(\mathcal{B})f|^{q}    \right)^{\frac{1}{q}}\right\Vert_{L^{1,\infty}(G)}\\
  &\leq    \left\Vert \left(\sum_{l=0}^{\infty} |A\psi_{l}(\mathcal{B})[\psi_{l-1}(\mathcal{B})+\psi_{l}(\mathcal{B})+\psi_{l+1}(\mathcal{B})] f|^{q}    \right)^{\frac{1}{q}}\right\Vert_{L^{1,\infty}(G)}\\
  &=\Vert W(\{[\psi_{l-1}(\mathcal{B})+\psi_{l}(\mathcal{B})+\psi_{l+1}(\mathcal{B})] f\}_{l=0}^{\infty}) \Vert_{L^{1,\infty}(\ell^q)}\\
  &\lesssim \Vert \{ [\psi_{l-1}(\mathcal{B})+\psi_{l}(\mathcal{B})+\psi_{l+1}(\mathcal{B})] f\}_{l=0}^{\infty}) \Vert_{L^1(\ell^q)} \lesssim \Vert f \Vert_{F^{0}_{1,q}(G)}. 
\end{align*}
By observing that Lemmas \ref{Lemma1} and \ref{Lemma2} are enough for proving Theorem \ref{CardonaRuzhansky}, we will proceed with their proofs. For this, let us recall that on any amenable topological group, and also in several spaces of homogeneous type one has the Calder\'on-Decomposition lemma. We fix it in the following remark.
\end{remark}
\begin{remark}\label{CZdecompositioncompact} For any   non-negative function $f\in  L^1(G),$ one has  its Calder\'on-Zygmund decomposition. Indeed, by following Hebish \cite{Hebish}, (whose construction remains valid for any amenable group, in particular, compact Lie groups)  one can  obtain a suitable family of  disjoint open sets  $\{I_j\}_{j=0}^{\infty}$  such that
\begin{itemize}
    \item $f(x)\leq t,$ for $a.e.$ $x\in G\setminus \cup_{j\geq 0}I_j,$\\
    
    \item $\sum_{j\geq 0}|I_j|\leq \frac{C}{t}\Vert f\Vert_{L^1(G)},$ and\\ 
    
    \item $t|I_j|\leq \int_{I_j}f(x)dx\leq 2|I_j|t,$ for all $j.$ 
\end{itemize}Now, for every $j\in \mathbb{N}_0,$ let us define $R_j$ by
\begin{equation}\label{Rj}
    R_{j}:=\sup\{R>0: B(z_j,R)\subset I_j, \textnormal{   for some   }z_j\in I_j\},
\end{equation} where $B(z_j,R)=\{x\in I_j:|z_j^{-1}x|<R\}.$ Every $I_j$ is bounded, and one can assume that  $I_j\subset B(z_j,2R_j),$ where $z_j\in I_j.$ 
Let us note that by assuming $f(e_G)>t,$ (this  can be done, just by re-defining $f\in  L^1(G)$ at the identity element $e_G$ of $G$) we should have that
\begin{equation}\label{eGasump}
    e_{G}\in \bigcup_{j}I_j,
\end{equation} because $f(x)\leq t,$ for $a.e.$ $x\in G\setminus \cup_{j\geq 0}I_j.$

Let us define, for every $x\in I_j,$
\begin{equation}
    g(x):=\frac{1}{|I_j|}\int\limits_{I_j}f(y)dy,\,\,\,b(x)=f(x)-g(x),
\end{equation} and for $x\in G\setminus\cup_{j\geq 0}I_j,$
$
    g(x)=f(x),\,\,\, b(x)=0.
$ Then, one has the decomposition $f=g+b,$ with $b$ having null average on $I_j.$

\end{remark}
\begin{proof}[Proof of Lemma \ref{Lemma1}] Let us fix $f\in L^1(G)$ to be a non-negative function, and let us consider its Calder\'on-Zygmund decomposition $f=g+b$ as in Remark \ref{CZdecompositioncompact}. It is enough to demonstrate that the linear operators $W_{\ell},$ $\ell\in \mathbb{N}_0,$ are uniformly bounded on $L^q(G),$ $1<q<\infty.$ This fact is straightforward  if $q=2,$ so it is suffices (by the duality argument) that the operators $W_{\ell}$ are uniformly bounded from $L^1(G)$ into $L^{1,\infty}(G).$ So, we will prove the existence of $C>0,$ independent of $f\in L^1(G),$ and  $\ell\in \mathbb{N}_0,$ such that
\begin{equation}\label{weak(1,1)inequalityWell}
    \left|\left\{x\in G:|W_\ell f(x)|  >t \right\}\right|\leq \frac{C}{t}\Vert f\Vert_{L^1(G)}.
\end{equation}
 Let us remark  that for every $x\in I_j,$ 
\begin{align*}
    |g(x)|= \left|   \frac{1}{|I_j|}\int\limits_{I_j}f(y)dy \right|\leq 2t.
\end{align*}
By  the Minkowski inequality, we have \begin{align*}
   & \left|\left\{x\in G:|W_{\ell} f(x)|>t \right\}\right|\leq \left|\left\{x\in G:|W_{\ell} g(x)|>\frac{t}{2} \right\} \right|\\
    &\hspace{3cm}+\left|\left\{x\in G:|W_{\ell} b(x)|>\frac{t}{2} \right\}\right|.
\end{align*}
In view of the Chebyshev inequality, we get
\begin{align*}
    & \left|\left\{x\in G:|W_\ell f(x)|>t \right\}\right|\\
    & \leq \left|\left\{x\in G:|W_{\ell} g(x)|>\frac{t}{2} \right\} \right|+\left|\left\{x\in G:|W_{\ell} b(x)|>\frac{t}{2} \right\}\right|\\
    &=  \left|\left\{x\in G:|W_{\ell} g(x)|^2>\frac{t^{2}}{2^2} \right\} \right|+\left|\left\{x\in G:|W_{\ell} b(x)|>\frac{t}{2} \right\}\right|\\
    &\leq \frac{2^2}{t^2}\int\limits_{G}|W_\ell g(x)|^2dx+ \left|\left\{x\in G:|W_{\ell}b(x)|>\frac{t}{2} \right\}\right|\\
    &\leq \frac{2^2}{t^2}\sup_{\ell}\Vert W_{\ell}\Vert_{\mathscr{B}(L^2(G))}\int\limits_{G}| g(x)|^2dx+ \left|\left\{x\in G:|W_{\ell}b(x)|>\frac{t}{2} \right\}\right|\\
    &\lesssim \frac{2^2}{t^2}\int\limits_{G}| g(x)|^2dx+ \left|\left\{x\in G:|W_{\ell}b(x)|>\frac{t}{2} \right\}\right|,
\end{align*}
in view of the $L^2(G)$-boundedness of $A$ and the fact that the operators $\psi_{\ell}(\mathcal{B})$ are $L^2(G)$-bounded uniformly  in $\ell\in \mathbb{N}_0.$ Additionally, note that the  estimate
\begin{align*}
    \Vert g\Vert_{L^2(G)}^2&=\int\limits_{G}|g(x)|^2dx=\sum_{j}\int\limits_{I_j}|g(x)|^2dx+\int\limits_{G\setminus \cup_{j}I_j}|g(x)|^2dx\\
    &=\sum_{j}\int\limits_{I_j}|g(x)|^2dx+\int\limits_{G\setminus \cup_{j}I_j}|f(x)|^2dx\\
    &\leq \sum_{j}\int\limits_{I_j}(2t)^{2}dx+\int\limits_{G\setminus \cup_{j}I_j}f(x)^2dx\lesssim t^{2}\sum_{j}|I_j|+\int\limits_{G\setminus \cup_{j}I_j}f(x)f(x)dx\\
    &\leq t^{2}\times \frac{C}{t}\Vert f\Vert_{L^1(G)}+t\int\limits_{G\setminus \cup_{j}I_j}f(x)dx\lesssim t\Vert f\Vert_{L^1(G)},
\end{align*} implies that
\begin{align*}
  \Small{ \left|\left\{x\in G:|W_\ell f(x)|>t \right\}\right|\leq \frac{4}{t}\Vert f\Vert_{L^1(G)}+ \left|\left\{x\in G:|W_\ell b(x)|^2   >\frac{t}{2} \right\}\right|.}
\end{align*}
Taking into account that $b\equiv 0$ on $G\setminus \cup_j I_j,$ we have that
\begin{equation}
    b=\sum_{k}b_k,\,\,\,b_k(x)=b(x)\cdot 1_{I_k}(x).
\end{equation} Let us assume that $I_{j}^*$ is an open set, such that $I_j\subset I_j^*,$ and $|I_{j}^*|=K|I_{j}|$ for some $K>0,$ and $\textnormal{dist}(\partial I_{j}^*,\partial I_{j})\geq 4c\,\textnormal{dist}(\partial I_{j},e_{G}),$ where $c>0$ and $e_{G}$ is the identity element of $G$.  So, by the Minkowski inequality we have
\begin{align*}
     & \left|\left\{x\in G:|W_\ell b(x)|  >\frac{t}{2} \right\}\right|\\
      &=\left|\left\{x\in \cup_j I_j^*:|W_\ell b(x)|   >\frac{t}{2} \right\}\right|+\left|\left\{x\in G\setminus  \cup_j I_j^*:|W_\ell b(x)|   >\frac{t}{2} \right\}\right|\\
      &\leq \left|\left\{x\in G:x\in \cup_j I_j^* \right\}\right|+\left|\left\{x\in G\setminus  \cup_j I_j^*:|W_\ell b(x)|   >\frac{t}{2} \right\}\right|.
      \end{align*} Consequently, we deduce the estimates
      \begin{align*} & \left|\left\{x\in G:|W_\ell b(x)|  >\frac{t}{2} \right\}\right|\leq\sum_{j}|I_j^*|+\left|\left\{x\in G\setminus  \cup_j I_j^*:|W_\ell b(x)|   >\frac{t}{2} \right\}\right|\\
      &=K\sum_{j}|I_j|+\left|\left\{x\in G\setminus  \cup_j I_j^*:|W_\ell b(x)|   >\frac{t}{2} \right\}\right|\\
      &\leq \frac{CK}{t}\Vert f\Vert_{L^1(G)}+\left|\left\{x\in G\setminus  \cup_j I_j^*:|W_\ell b(x)|   >\frac{t}{2} \right\}\right|.
  \end{align*}Now, using the Chebyshev inequality  to estimate the right hand side above we obtain
  \begin{align*}
     & \left|\left\{x\in G\setminus  \cup_j I_j^*:|W_\ell b(x)|   >\frac{t}{2} \right\}\right|\leq\frac{2}{t}\int\limits_{ G\setminus  \cup_j I_j^*} |W_{\ell}b(x)|dx\\
     &\leq  \frac{2}{t}\sum_{k}\int\limits_{ G\setminus  \cup_j I_j^*} |W_{\ell}b_k(x)|dx.
  \end{align*}
From now, let us denote by $\kappa_\ell$  the right convolution kernel of     $W_\ell:=A\psi_\ell(\mathcal{B}) .$ Observe that
\begin{align*}
     &\left|\left\{x\in G\setminus  \cup_j I_j^*:|W_\ell b(x)|   >\frac{t}{2} \right\}\right|\leq \frac{2}{t}\sum_{k}\int\limits_{ G\setminus  \cup_j I_j^*} |W_{\ell}b_k(x)|dx\\
     &= \frac{2}{t}\sum_{k}\int\limits_{ G\setminus  \cup_j I_j^*}\left|b_k   \ast \kappa_{\ell}(x)\right|dx\\
     &= \frac{2}{t}\sum_{k}\int\limits_{ G\setminus  \cup_j I_j^*}\left|\int\limits_{I_k}b_k(z)\kappa_{\ell}(z^{-1}x)dz   \right|dx.
     \end{align*} By using that the average of $b_{k}$ on $I_k$ is zero, $\int_{I_k}b_{k}(z)dz=0,$ we have
     \begin{align*}
     \frac{2}{t}\sum_{k} \int\limits_{ G\setminus  \cup_j I_j^*}  &\left|\int\limits_{I_k}b_k(z)\kappa_{\ell}(z^{-1}x)dz   \right|dx\\
      &=\frac{2}{t}\sum_{k}\int\limits_{ G\setminus  \cup_j I_j^*}\left|\int\limits_{I_k}b_k(z)\kappa_{\ell}(z^{-1}x)dz  - \kappa_{\ell}(x)\int\limits_{I_k}b_{k}(z)dz \right|dx\\
       &=\frac{2}{t}\sum_{k}\int\limits_{ G\setminus  \cup_j I_j^*}\left|\int\limits_{I_k}(\kappa_{\ell}(z^{-1}x)  - \kappa_{\ell}(x))b_{k}(z) dz\right|dx.
     \end{align*}
 Assuming the following uniform estimate
 \begin{equation}\label{calderonkernels}
    M=\sup_{k} \sup_{z\in I_k}\sum_{\ell=0}^\infty\int\limits_{ G\setminus  \cup_j I_j^*}|\kappa_{\ell}(z^{-1}x)  - \kappa_{\ell}(x)|dx<\infty,
 \end{equation} we have
     \begin{align*}
          & \left|\left\{x\in G\setminus  \cup_j I_j^*:|W_\ell b(x)|   >\frac{t}{2} \right\}\right|\leq \frac{2M}{t}\sum_{k}\int\limits_{I_k}|b_{k}(z)|dz\\
          &=\frac{2M}{t}\| b\|_{L^1(G)}\leq \frac{6M}{t}\| f\|_{L^1(G)}.
     \end{align*}
     So, if we prove the uniform estimate \eqref{calderonkernels}, we obtain the weak (1,1) inequality  for $f\in  L^1(G),$ $f\geq 0$. For the proof of \eqref{calderonkernels} let us use the estimates of the Calder\'on-Zygmund kernel of every operator $W_{\ell}.$ 
     Before continuing with the proof let us use the following geometrical property of the open sets $I_{j},$ $j\geq 0$:
     \begin{lemma}\label{distanceboundaries}Let us consider  \eqref{eGasump}. Then, with the notation above, for any $ x\in G\setminus \cup_jI_j^*,$ and   for all  $z\in {I_k},$ we have $|z|\lesssim |x|,$ i.e. for some $c>0,$ $  4c|z|\leq |x|.$
\end{lemma}
     We will postpone the proof of Lemma \ref{distanceboundaries} for a moment to  continue with the proof of Lemma \ref{Lemma1}, in order to  use  the estimate $  4c|z|\leq |x| ,$ for $ x\in G\setminus \cup_jI_j^*,$     $z\in {I_k},$ together with   \eqref{3.8} in Lemma \ref{lemmaofauxiliaruse} for $\ell\geq 0,$ as follows, 
     \begin{align*}
        M_k&:= \sup_{z\in I_k}\sum_{\ell=0}^\infty\int\limits_{ G\setminus  \cup_j I_j^*}  |\kappa_{\ell}(z^{-1}x)  - \kappa_{\ell}(x)|dx\lesssim \sup_{z\in I_k}\sum_{\ell=0}^\infty\int\limits_{ |x|\geq 4c|z|}  |\kappa_{\ell}(z^{-1}x)  - \kappa_{\ell}(x)|dx\\
        &\lesssim \sum_{\ell=0}^\infty  2^{-\ell(s-\frac{n}{2})}\Vert\sigma \Vert_{l.u.,\, L_{s}^2(\widehat{G})}'.
     \end{align*}The estimate above implies  that 
     \begin{equation}\label{fundamentalbound}
         M_k:= \sup_{z\in I_k}\sum_{\ell=0}^\infty\int\limits_{ G\setminus  \cup_j I_j^*}  |\kappa_{\ell}(z^{-1}x)  - \kappa_{\ell}(x)|dx\lesssim \Vert\sigma \Vert_{l.u.,\, L_{s}^2(\widehat{G})}'.
     \end{equation}
      So, we have  \eqref{weak(1,1)inequalityWell} for $f\in  L^1(G)$ with $f\geq 0.$
     Note that if $f\in  L^1(G)$ is real-valued, one can decompose $f=f^+-f^-,$ as the difference of two non-negative functions, where $f^+,f^-\in  L^1(G), $ and $|f|=f^++f^{-}.$ Because $f^+,f^-\leq |f|,$ we have
     \begin{align*}
   \left|\left\{x\in G:|W_\ell f(x)|  >t \right\}\right|
    &\leq  \left|\left\{x\in G:|W_\ell f_+(x)|  >\frac{t}{2} \right\}\right| +\left|\left\{x\in G:|W_\ell f_{-}(x)|  >\frac{t}{2} \right\}\right|\\
    &\leq  \frac{C}{t}\Vert f_+\Vert_{L^1(G)}+\frac{C}{t}\Vert f_{-}\Vert_{L^1(G)}\\
    &\leq  \frac{2C}{t}\Vert f\Vert_{L^1(G)}.
\end{align*}A similar analysis, by splitting a complex function $f\in   L^1(G)$ into its real and imaginary parts allows to conclude the weak (1,1) inequality \eqref{weak(1,1)inequalityWell} to complex functions.
     Thus, the proof of Lemma \ref{Lemma1} would be complete if we prove Lemma \ref{distanceboundaries}. That lemma was proved in \cite[Page 17]{CardonaRuzhanskyTriebelGraded} for graded Lie groups and the same proof applies for any amenable group (see Hebish \cite{Hebish}),    we will  present the proof here for the case of a compact Lie group $G$ for completeness.
\end{proof}
\begin{proof}[Proof of Lemma \ref{distanceboundaries}]
In view of the estimate $\textnormal{dist}(\partial I_{j}^*,\partial I_{j})\geq 4c\textnormal{dist}(\partial I_{j},e_{G}),$ we will prove that for $x\in G\setminus \cup_jI_j^*,$   and all $z\in I_{k},$ $4c|z|=4c\times \textnormal{dist}(z,e_G)\lesssim  \textnormal{dist}(\partial I_{k}^*,\partial I_{k})  \leq|x|.$  Indeed, fix $\varepsilon>0,$ and let us take 
     $w\in \partial I_{k},$ and $w'\in \partial I_{k}^*$ such that $d(w,w')\leq \textnormal{dist}(\partial I_{k},\partial I_{k}^*)+\varepsilon. $ Then, from the triangle inequality, we have
    \begin{equation}\label{Estimateforboundaries} 
\begin{array}{l}
        d(z,e_G)\\
        \leq d(z,w)+d(w,w')+d(w',e_G)
         \leq \textnormal{diam}(I_k)+\textnormal{dist}(\partial I_{k},\partial I_{k}^*)+\textnormal{dist}(\partial I_k^*,e_G)+\varepsilon\\
         \lesssim \textnormal{diam}(I_k)+\textnormal{dist}(\partial I_{k},\partial I_{k}^*)+\textnormal{dist}(\partial I_{k},e_G)+\varepsilon\\
         \lesssim \textnormal{diam}(I_k)+\textnormal{dist}(\partial I_{k},\partial I_{k}^*)+\frac{1}{4c}\textnormal{dist}(\partial I_{k},\partial I_{k}^*)+\varepsilon\\
         \asymp \textnormal{dist}(\partial I_{k},\partial I_{k}^*)+\varepsilon,
\end{array}
\end{equation}
where in the last line we have assumed that $\textnormal{diam}(I_k)\asymp \textnormal{dist}(\partial I_{k},\partial I_{k}^*),$ (with constants of proportionality independent in $k$) and that $\textnormal{dist}(\partial I_{k},\partial I_{k}^*)$ is proportional to $R_k$ in view of the relation $|I_k^*|=K|I_k|.$ Assuming \eqref{Estimateforboundaries}, one has that for all $\varepsilon>0,$ $ d(z,e_G)\lesssim  \textnormal{dist}(\partial I_{k},\partial I_{k}^*)+\varepsilon,$ which implies that \begin{equation}\label{mainestimateofk}
    d(z,e_G)\lesssim  \textnormal{dist}(\partial I_{k},\partial I_{k}^*).
\end{equation}  To show  that the proportionality constant in \eqref{mainestimateofk} is uniform in $k,$  let us recall the definition of the radii $R_{k}'s$ in \eqref{Rj},
  that  $B(z_k,R_{k})\subset I_k\subset B(z_{k},2R_{k}),$ and that $B(z_k,R_{k}/C)\subset I_{k}^{*}\subset B(z_k, CR_{k})$ for some $C>2$ independent of $k,$ where for any $k,$ $z_k\in I_{k}.$ From this remark observe that:
  \begin{itemize}
      \item The condition  $B(z_k,R_{k})\subset I_k\subset B(z_{k},2R_{k}),$ implies that $2R_k\leq \textnormal{diam}(I_k)\leq 4R_{k}. $
      \item That $B(z_k,R_{k})\subset I_k\subset I_{k}^{*}\subset B(z_{k},CR_{k}),$ implies that $$ \textnormal{dist}(\partial I_{k},\partial I_{k}^*) \leq \textnormal{dist}(\partial  B(z_{k},R_{k}),\partial B(z_{k},CR_{k}))=(C-1)R_{k}.  $$ On the other hand, by observing that in every step above we can replace $I_{k}^{*}:=B(z_{k},CR_{k}),$ in view of the inclusion $$ I_k\subset B(z_{k},2R_{k}) \subset I_{k}^{*}:= B(z_{k},CR_{k}),$$ we have
      $$ (C-2)R_k=\textnormal{dist}(\partial I_{k}^{*},\partial B(z_{k},2R_{k}) ) \leq \textnormal{dist}(\partial I_{k},\partial I_{k}^*) . $$
  \end{itemize}Consequently,
 \begin{align*}
   \textnormal{diam}(I_k) &\asymp R_{k}\asymp \textnormal{dist}(\partial  B(z_{k},2R_{k}),\partial B(z_{k},CR_{k})) \\
   &\asymp \textnormal{dist}(\partial I_{k},\partial I_{k}^*).
 \end{align*} To show that $\textnormal{dist}(\partial I_{k}^*,\partial I_{k})  \leq|x|,$ observe that from Remark \ref{eGasump}, $e_{G}\in \cup_{j}I_j,$  and because of $x\in G\setminus \cup_{j}I_j,$  $$\textnormal{dist}(\partial I_{k}^*,\partial I_{k})\lesssim \textnormal{diam}(\cup_jI_j) \lesssim d(x,e_G)=|x|. $$
So, we have guaranteed the existence of a positive constant, which we again denote by $c>0,$ such that,  $\{x\in G: x\in G\setminus \cup_jI_j^*\}\subset\{x\in G:\textnormal{ for all } z\in {I_k},\,\,\, 4c|z|\leq |x| \}.$
\end{proof}
\begin{proof}[Proof of Lemma \ref{Lemma2}]
Now, we claim that 
\begin{equation}\label{weakvectorvalued}
    W:L^{1}(G,\ell^r(\mathbb{N}_0))\rightarrow L^{1,\infty}(G,\ell^r(\mathbb{N}_0)),\,\,\,1<r<\infty.
\end{equation}
extends to a bounded operator. For the proof of \eqref{weakvectorvalued}, we need to show that there exists a constant $C>0$ independent of $\{f_\ell\}\in L^1(G,\ell^r(\mathbb{N}_0))$ and $t>0,$ such that 
\begin{equation}
    \left|\left\{x\in G:\left(\sum_{\ell=0}^\infty|W_\ell f_\ell(x)|^r   \right)^{\frac{1}{r}}>t \right\}\right|\leq \frac{C}{t}\Vert \{f_\ell\} \Vert_{L^1(G,\ell^r(\mathbb{N}_0)}.
\end{equation} So, fix $\{f_\ell\}\in L^1(G,\ell^r(\mathbb{N}_0))$ and $t>0,$ and let $h(x):= \left(\sum_{\ell=0}^\infty| f_\ell(x)|^r   \right)^{\frac{1}{r}}.$ Apply the Calder\'on-Zygmund decomposition Lemma to $h\in L^1(G)$ in Remark \ref{CZdecompositioncompact}  in order to obtain a disjoint collection $\{I_j\}_{j=0}^{\infty}$ of disjoint open sets such that
\begin{itemize}
    \item $h(x)\leq t,$ for $a.e.$ $x\in G\setminus \cup_{j\geq 0}I_j,$\\
    
    \item $\sum_{j\geq 0}|I_j|\leq \frac{C}{t}\Vert h\Vert_{L^1(G)},$ and\\ 
    
    \item $t\leq \frac{1}{|I_j|}\int_{I_j}h(x)dx\leq 2t,$ for all $j.$
\end{itemize} Now, we will define a suitable decomposition of $f_\ell,$ for every $\ell\geq 0.$ Recall that   every $I_j$ is bounded, and that  $I_j\subset B(z_j,2R_j),$ where $z_j\in I_j$ (see Remark \ref{CZdecompositioncompact}).  Let us define, for every $\ell,$ and $x\in I_j,$
\begin{equation}
    g_\ell(x):=\frac{1}{|I_j|}\int\limits_{I_j}f_\ell(y)dy,\,\,\,b_{\ell}(x)=f_\ell(x)-g_{\ell}(x).
\end{equation} and for $x\in G\setminus\cup_{j\geq 0}I_j,$
\begin{equation}
    g_\ell(x)=f_\ell(x),\,\,\, b_\ell(x)=0.
\end{equation} So, for a.e.  $x\in G,$ $f_\ell(x)=g_\ell(x)+b_\ell(x).$
Note that for every $1<r<\infty,$ $\Vert \{g_\ell\}\Vert_{L^r(\ell^r)}^r\leq t^{r-1}\Vert \{f_\ell\} \Vert_{L^1(\ell^r)}. $ Indeed, for $x\in I_j$, Minkowski integral inequality gives,
\begin{align*}
    \left(\sum_{\ell=0}^\infty|g_\ell(x)|^r\right)^{\frac{1}{r}} &\leq  \left(\sum_{\ell=0}^\infty\left|   \frac{1}{|I_j|}\int\limits_{I_j}f_\ell(y)dy \right|^r\right)^{\frac{1}{r}}\leq \frac{1}{|I_j|}\int\limits_{I_j}\left(\sum_{\ell=0}^\infty|   f_\ell(y) |^r\right)^{\frac{1}{r}}dy\\
    &=\frac{1}{|I_j|}\int\limits_{I_j}h(y)dy\\
    &\leq 2t.
\end{align*}
Consequently,
\begin{align*}
    \sum_{\ell=0}^\infty|g_\ell(x)|^r\leq (2t)^{r},
\end{align*}
and from the fact that $h(x)\leq t,$ for $a.e.$ $x\in G\setminus \cup_{j\geq 0}I_j,$ we have
\begin{align*}
    \Vert \{g_\ell\}\Vert_{L^r(\ell^r)}^r&=\int\limits_{G}\sum_{\ell=0}^\infty|g_\ell(x)|^rdx=\sum_{j}\int\limits_{I_j}\sum_{\ell=0}^\infty|g_\ell(x)|^rdx+\int\limits_{G\setminus \cup_{j}I_j}\sum_{\ell=0}^\infty|g_\ell(x)|^rdx\\
    &=\sum_{j}\int\limits_{I_j}\sum_{\ell=0}^\infty|g_\ell(x)|^rdx+\int\limits_{G\setminus \cup_{j}I_j}\sum_{\ell=0}^\infty|f_\ell(x)|^rdx\\
    &\leq \sum_{j}\int\limits_{I_j}(2t)^{r}dx+\int\limits_{G\setminus \cup_{j}I_j}h(x)^rdx\\
    &\lesssim t^{r}\sum_{j}|I_j|+\int\limits_{G\setminus \cup_{j}I_j}h(x)^{r-1}h(x)dx\\
    &\leq t^{r}\times \frac{C}{t}\Vert h\Vert_{L^1(G)}+t^{r-1}\int\limits_{G\setminus \cup_{j}I_j}h(x)dx\lesssim t^{r-1}\Vert h\Vert_{L^1(G)}\\
    &=t^{r-1}\Vert \{f_\ell\} \Vert_{L^1(\ell^r)}.
\end{align*}
Now, by using   the Minkowski and the Chebyshev inequality, we obtain 
\begin{align*}
    & \left|\left\{x\in G:\left(\sum_{\ell=0}^\infty|W_{\ell} f_\ell(x)|^r   \right)^{\frac{1}{r}}>t \right\}\right|\\
    &\leq  \left|\left\{x\in G:\left(\sum_{\ell=0}^\infty|W_{\ell} g_\ell(x)|^r   \right)^{\frac{1}{r}}>\frac{t}{2} \right\}\right|+ \left|\left\{x\in G:\left(\sum_{\ell=0}^\infty|W_{\ell} b_\ell(x)|^r   \right)^{\frac{1}{r}}>\frac{t}{2} \right\}\right|\\
    &\leq \frac{2^r}{t^r}\int\limits_{G}\sum_{\ell=0}^\infty|W_{\ell} g_\ell(x)|^rdx+ \left|\left\{x\in G:\left(\sum_{\ell=0}^\infty|W_{\ell} b_\ell(x)|^r   \right)^{\frac{1}{r}}>\frac{t}{2} \right\}\right|.
\end{align*} In view of Lemma \ref{Lemma1}, $W:L^r(G,\ell^r(\mathbb{N}_0))\rightarrow L^r(G,\ell^r(\mathbb{N}_0)),$ extends to a bounded operator and 
\begin{equation}
    \int\limits_{G}\sum_{\ell=0}^\infty|W_{\ell} g_\ell(x)|^rdx=\Vert W\{g_\ell\}\Vert_{L^r(\ell^r)}^r\lesssim  \Vert \{g_\ell\}\Vert_{L^r(\ell^r)}^r\leq t^{r-1}\Vert \{f_\ell\}\Vert_{L^1(\ell^r)}.
\end{equation}
Consequently,
 \begin{align*}
    & \left|\left\{x\in G:\left(\sum_{\ell=0}^\infty|W_{\ell} f_\ell(x)|^r   \right)^{\frac{1}{r}}>t \right\}\right|\\
    &\lesssim  \frac{1}{t}\Vert \{f_\ell\}\Vert_{L^1(\ell^r)}+ \left|\left\{x\in G:\left(\sum_{\ell=0}^\infty|W_{\ell} b_\ell(x)|^r   \right)^{\frac{1}{r}}>\frac{t}{2} \right\}\right|.
\end{align*}
Now, we only need to prove that
\begin{equation}
     \left|\left\{x\in G:\left(\sum_{\ell=0}^\infty|W_{\ell} b_\ell(x)|^r   \right)^{\frac{1}{r}}>\frac{t}{2} \right\}\right|\lesssim  \frac{1}{t}\Vert \{f_\ell\}\Vert_{L^1(\ell^r)}.   
\end{equation} 
Taking into account that $b_\ell\equiv 0$ on $G\setminus \cup_j I_j,$ we have that
\begin{equation}
    b_\ell=\sum_{k}b_{\ell,k},\,\,\,b_{\ell,k}(x)=b_{\ell}(x)\cdot 1_{I_k}(x).
\end{equation} Let us assume that $I_{j}^*$ is a open set, such that $|I_{j}^*|=K|I_{j}|$ for some $K>0,$ and $\textnormal{dist}(\partial I_{j}^*,\partial I_{j})\geq c\textnormal{dist}(\partial I_{j},e_{G}),$  where $e_{G}$ is the identity element of $G$.  So, by the Minkowski inequality we have,
\begin{align*}
     & \left|\left\{x\in G:\left(\sum_{\ell=0}^\infty|W_{\ell} b_{\ell}(x)|^r   \right)^{\frac{1}{r}}>\frac{t}{2} \right\}\right|\\
      &=\left|\left\{x\in \cup_j I_j^*:\left(\sum_{\ell=0}^\infty|W_{\ell} b_\ell(x)|^r   \right)^{\frac{1}{r}}>\frac{t}{2} \right\}\right|+\left|\left\{x\in G\setminus  \cup_j I_j^*:\left(\sum_{\ell=0}^\infty|W_{\ell} b_\ell(x)|^r   \right)^{\frac{1}{r}}>\frac{t}{2} \right\}\right|\\
      &\leq \left|\left\{x\in G:x\in \cup_j I_j^* \right\}\right|+\left|\left\{x\in G\setminus  \cup_j I_j^*:\left(\sum_{\ell=0}^\infty|W_{\ell} b_\ell(x)|^r   \right)^{\frac{1}{r}}>\frac{t}{2} \right\}\right|.
      \end{align*} Taking into account that $$  \left|\left\{x\in G:x\in \cup_j I_j^* \right\}\right| \leq \sum_{j}|I_j^*|,  $$ we get
      \begin{align*}  & \left|\left\{x\in G:\left(\sum_{\ell=0}^\infty|W_{\ell} b_\ell(x)|^2   \right)^{\frac{1}{2}}>\frac{t}{2} \right\}\right|\\
      &\leq\sum_{j}|I_j^*|+\left|\left\{x\in G\setminus  \cup_j I_j^*:\left(\sum_{\ell=0}^\infty|W_{\ell} b_\ell(x)|^2   \right)^{\frac{1}{2}}>\frac{t}{2} \right\}\right|\\
      &=K\sum_{j}|I_j|+\left|\left\{x\in G\setminus  \cup_j I_j^*:\left(\sum_{\ell=0}^\infty|W_{\ell} b_\ell(x)|^2   \right)^{\frac{1}{2}}>\frac{t}{2} \right\}\right|\\
      &\leq \frac{CK}{t}\Vert f\Vert_{L^1(G,\ell^r)}+\left|\left\{x\in G\setminus  \cup_j I_j^*:\left(\sum_{\ell=0}^\infty|W_{\ell} b_\ell(x)|^2   \right)^{\frac{1}{2}}>\frac{t}{2} \right\}\right|.
  \end{align*} Observe that  the Chebyshev inequality implies
  \begin{align*}
     & \left|\left\{x\in G\setminus  \cup_j I_j^*:\left(\sum_{\ell=0}^\infty|W_{\ell} b_\ell(x)|^r  \right)^{\frac{1}{r}}>\frac{t}{2} \right\}\right|\\
     &\leq\frac{2}{t}\int\limits_{ G\setminus  \cup_j I_j^*} \left(\sum_{\ell=0}^\infty|W_{\ell} b_\ell(x)|^r   \right)^{\frac{1}{r}}dx\\
     &=\frac{2}{t}\int\limits_{ G\setminus  \cup_j I_j^*} \left(\sum_{\ell=0}^\infty\left|\left(W_{\ell}\left(\sum_{k} b_{\ell,k}\right)   \right)(x)\right|^r\right)^{\frac{1}{r}}dx\\
     &=\frac{2}{t}\int\limits_{ G\setminus  \cup_j I_j^*} \Vert\{(W_{\ell}(\sum_{k} b_{\ell,k}) (x)\}_{\ell=0}^\infty\Vert_{\ell^r(\mathbb{N}_0)} dx\\
     &=\frac{2}{t}\int\limits_{ G\setminus  \cup_j I_j^*} \Vert\{\sum_{k} (W_{\ell}b_{\ell,k})(x) \}_{\ell=0}^\infty\Vert_{\ell^r(\mathbb{N}_0)} dx\\
     &\leq \frac{2}{t}\sum_{k}\int\limits_{ G\setminus  \cup_j I_j^*} \left(\sum_{\ell=0}^\infty\left|\left(W_{\ell}b_{\ell,k}   \right)(x)\right|^r\right)^{\frac{1}{r}}dx.
  \end{align*}
Now, if $\kappa_\ell$ is the right convolution Calder\'on-Zygmund kernel of     $W_{\ell} ,$  and by using that $\int_{I_k}b_{k,\ell}(y)dy=0,$ we have that
\begin{align*}
    \left(\sum_{\ell=0}^\infty\left|\left(W_{\ell}b_{\ell,k}   \right)(x)\right|^r\right)^{\frac{1}{r}}&=\left(\sum_{\ell=0}^\infty\left|b_{\ell,k}\ast \kappa_{\ell}(x)   \right|^r\right)^{\frac{1}{r}}\\
    &=\left(\sum_{\ell=0}^\infty\left| \int\limits_{I_k}\kappa_\ell(y^{-1}x)b_{\ell,k}(y)dy-\kappa_{\ell}(x)\int\limits_{I_k}b_{\ell,k}(y)dy \right|^r\right)^{\frac{1}{r}}\\
       &=\left(\sum_{\ell=0}^\infty\left| \int\limits_{I_k}(\kappa_\ell(y^{-1}x)-\kappa_{\ell}(x))b_{\ell,k}(y)dy \right|^r\right)^{\frac{1}{r}}.
\end{align*} Now, we will proceed as follows. By using that
$
    |b_{\ell,k}(y)|^r\leq \sum_{\ell'=0}^{\infty}|b_{\ell',k}(y)|^r,
$ by an application of the Minkowski integral inequality, we have
\begin{align*}
    &\left(\sum_{\ell=0}^\infty\left|\left(W_{\ell}b_{\ell,k}   \right)(x)\right|^r\right)^{\frac{1}{r}}= \left(\sum_{\ell=0}^\infty\left| \int\limits_{I_k}(\kappa_\ell(y^{-1}x)-\kappa_{\ell}(x))b_{\ell,k}(y)dy \right|^r\right)^{\frac{1}{r}}\\
    &\leq \int\limits_{I_k}\left(   \sum_{\ell=0}^\infty |\kappa_\ell(y^{-1}x)-\kappa_{\ell}(x)|^r|b_{\ell,k}(y)|^r\right)^{\frac{1}{r}}dy\\
    &\leq  \int\limits_{I_k}\left(   \sum_{\ell'=0}^{\infty}|b_{\ell',k}(y)|^r\right)^{\frac{1}{r}}\left(   \sum_{\ell=0}^\infty |\kappa_\ell(xy^{-1})-\kappa_{\ell}(x)|^r\right)^{\frac{1}{r}}dy.
\end{align*} 
Consequently, we deduce,
\begin{align*}
   & \frac{2}{t}\sum_{k}\int\limits_{ G\setminus  \cup_j I_j^*} \left(\sum_{\ell=0}^\infty\left|\left(W_{\ell}b_{\ell,k}   \right)(x)\right|^r\right)^{\frac{1}{r}}dx\\
   &\leq \frac{2}{t} \sum_{k}\int\limits_{ G\setminus  \cup_j I_j^*} \int\limits_{I_k}\left(   \sum_{\ell'=0}^{\infty}|b_{\ell',k}(y)|^r\right)^{\frac{1}{r}}\left(   \sum_{\ell=0}^\infty |\kappa_\ell(y^{-1}x)-\kappa_{\ell}(x)|^r\right)^{\frac{1}{r}}dy                   dx\\
   &= \frac{2}{t} \sum_{k} \int\limits_{I_k}  \int\limits_{ G\setminus  \cup_j I_j^*} \left(   \sum_{\ell'=0}^{\infty}|b_{\ell',k}(y)|^r\right)^{\frac{1}{r}}\left(   \sum_{\ell=0}^\infty |\kappa_\ell(y^{-1}x)-\kappa_{\ell}(x)|^r\right)^{\frac{1}{r}}dx                  dy\\
    &= \frac{2}{t} \sum_{k} \int\limits_{I_k}  \left(   \sum_{\ell'=0}^{\infty}|b_{\ell',k}(y)|^r\right)^{\frac{1}{r}}  \int\limits_{ G\setminus  \cup_j I_j^*} \left(   \sum_{\ell=0}^\infty |\kappa_\ell(y^{-1}x)-\kappa_{\ell}(x)|^r\right)^{\frac{1}{r}}dxdy.
\end{align*}
Let us recall that for $x\in G\setminus \cup_jI_j^*,$  and  $y\in I_{k},$ in view of Lemma \ref{distanceboundaries}, we have that $4c|y|  \leq|x|.$ So,  $$\{x\in G: x\in G\setminus \cup_jI_j^*\}\subset\{x\in G:\textnormal{ for all } z\in {I_k},\,\,\, 4c|z|\leq |x| \}.$$ Now, in view of  \eqref{3.8} in Lemma \ref{lemmaofauxiliaruse} with $\ell\geq 0,$ we deduce \eqref{fundamentalbound} and as a consequence we get 
\begin{align*}
  & \int\limits_{ G\setminus  \cup_j I_j^*} \left(   \sum_{\ell=0}^\infty |\kappa_\ell(y^{-1}x)-\kappa_{\ell}(x)|^r\right)^{\frac{1}{r}}dx \leq \int\limits_{ G\setminus  \cup_j I_j^*} \sum_{\ell=0}^\infty |\kappa_\ell(y^{-1}x)-\kappa_{\ell}(x)|dx\\
   &\leq  \sum_{\ell=0}^\infty \int\limits_{ G\setminus  \cup_j I_j^*}|\kappa_\ell(y^{-1}x)-\kappa_{\ell}(x)|dx\\\
   &\leq   \sum_{\ell=0}^\infty \int\limits_{|x| >4c|y|}|\kappa_\ell( y^{-1}x)-\kappa_{\ell}( x)|dx\lesssim \Vert\sigma \Vert_{l.u.,\, L_{s}^2(\widehat{G})}'.
\end{align*}Thus, we have proved the estimate
\begin{align*}
      &\left|\left\{x\in G:\left(\sum_{\ell=0}^\infty|W_{\ell} b_\ell(x)|^r   \right)^{\frac{1}{r}}>\frac{t}{2} \right\}\right|\lesssim  \frac{2}{t} \sum_{k} \int\limits_{I_k}  \left(   \sum_{\ell'=0}^{\infty}|b_{\ell',k}(y)|^r\right)^{\frac{1}{r}}dy \\
      &=\frac{2}{t}  \int\limits_{\cup_{k}I_{k}}  \left(   \sum_{\ell'=0}^{\infty}|b_{\ell'}(y)|^r\right)^{\frac{1}{r}}dy \\
      &\lesssim \frac{1}{t}\Vert \{f_\ell\}\Vert_{L^1(\ell^r)}. 
\end{align*} Thus, the proof of the weak (1,1) inequality is complete and we have that \begin{equation}\label{weakvectorvalued2}
    W:L^{1}(G,\ell^r(\mathbb{N}_0))\rightarrow L^{1,\infty}(G,\ell^r(\mathbb{N}_0)),\,\,\,1<r<\infty,
\end{equation} admits a bounded extension. The proof of Lemma \ref{Lemma2} is complete.
\end{proof}
Now, having proved Lemma \ref{Lemma1} and Lemma \ref{Lemma2},
 the duality argument in Remark \ref{RemarkInterpolation} proves the H\"ormander-Mihlin Theorem \ref{CardonaRuzhansky}. 
\end{proof}
\begin{remark}\label{finalreamrk}Let us return to the historical  Marcinkiewicz  condition on the torus $\mathbb{T},$
\begin{equation}\label{Marc12}
\sup_{\xi\in\mathbb{Z}}|\sigma(\xi) |+ \sup_{j\in\mathbb{N}_0}\sum_{2^{j-1}\leq |\xi|<2^{j}}|\sigma(\xi+1)-\sigma(\xi)|<\infty.
\end{equation}Let $s_0\in \mathbb{N}.$ For a general compact Lie group $G$ of dimension $n,$ one says that a Fourier multiplier $A\equiv T_\sigma$ satisfies the (weak) Marcinkiewicz condition of order $s_0,$ if its symbol $\sigma$ satisfies:
\begin{equation}\label{MarCond}
  \Vert\sigma\Vert_{L^{\infty}(\widehat{G})}+  \left\Vert \sigma \cdot 1_{\{[\xi]\in \widehat{G}: \,2^{j-1}\leq \langle\xi\rangle < 2^{j}  \}}\right\Vert_{\dot{L}^{1}_{s_0}(\widehat{G})}\lesssim_{s_0} 2^{j(n-s_0)},\,j\geq 0,
\end{equation} uniformly in $j,$ with the $\dot{L}^{1}_{s_0}(\widehat{G})$-norm defined via:
$$ \Vert \tau\Vert_{\dot{L}^{1}_{s_0}(\widehat{G})}:=\sum_{|\alpha|=s_0}\sum_{[\xi]\in \widehat{G}}d_\xi\textnormal{Tr}|\Delta^{\alpha}_\xi\tau(\xi)|,\,\,|\Delta^{\alpha}_\xi\tau|:=\sqrt{\textnormal{Tr}(\Delta^{\alpha}_\xi\tau(\Delta^{\alpha}_\xi\tau)^{*})}.$$ It was proved in \cite{FischerDiff}, that for $0\leq s_0\leq n,$ a symbol satisfying \eqref{MarCond} also satisfies the H\"ormander-Mihlin condition \eqref{toverifyHM} for any $s>n/2.$ In view of Theorem  \ref{CardonaRuzhansky}, a Fourier multiplier $A\equiv T_{\sigma}$ satisfying  \eqref{MarCond} extends to a bounded operator from  $F^{r}_{p,q}(G)$ into $F^{r}_{p,q}(G)$  for all $1<p,q<\infty,$ and all $r\in \mathbb{R},$ and  for $p=1,$ $A$ admits a bounded extension from $F^{r}_{1,q}(G)$ into $\textrm{weak-}F^{r}_{1,q}(G).$ In  the case of the torus $G=\mathbb{T},$ $1<p,q<\infty,$ and  $n=1=s_0,$ \eqref{MarCond} is just \eqref{Marc12}, and $A\equiv T_{\sigma}$ extends to a bounded operator   from  $F^{r}_{p,q}(\mathbb{T})$ into $F^{r}_{p,q}(\mathbb{T})$  for all $1<p,q<\infty,$ and all $r\in \mathbb{R},$ and  for $p=1,$ $A$ admits a bounded extension from $F^{r}_{1,q}(\mathbb{T})$ into $\textrm{weak-}F^{r}_{1,q}(\mathbb{T}).$ In view of the Littlewood-Paley theorem, for $1<p<\infty,$ $q=2$ and $r=0,$ the previous estimate  recovers the classical Marcinkiewicz estimate \cite{Marc}. We summarise this discussion in the following corollary.
\end{remark}
\begin{corollary}\label{HMTTL22}
Let us assume that $G$ is a compact Lie group of dimension $n.$  Let  $\sigma\in \Sigma(\widehat{G})$ be a  symbol satisfying 
\begin{equation}\label{MarCond2222}
  \Vert\sigma\Vert_{L^{\infty}(\widehat{G})}+  \left\Vert \sigma \cdot 1_{\{[\xi]\in \widehat{G}: \,2^{j-1}\leq \langle\xi\rangle < 2^{j}  \}}\right\Vert_{\dot{L}^{1}_{s_0}(\widehat{G})}\lesssim 2^{j(n-s_0)},\,j\geq 0,
\end{equation} uniformly in $j,$ for some $0\leq s_0\leq n,$ $s_0\in \mathbb{N}.$
Then $A\equiv T_\sigma$ extends to a bounded operator from  $F^{r}_{p,q}(G)$ into $F^{r}_{p,q}(G)$  for all $1<p,q<\infty,$ and all $r\in \mathbb{R}.$ For $p=1,$ $A$ admits a bounded extension from $F^{r}_{1,q}(G)$ into $\textrm{weak-}F^{r}_{1,q}(G).$   
\end{corollary}
\begin{proof}
That a symbol satisfying \eqref{MarCond2222} for some integer $s_0$ with $0\leq s_0\leq n,$ satisfies also the H\"ormander condition \eqref{toverifyHM} for all $s>\frac{n}{2}$  was proved in Theorem 6.16 of \cite{FischerDiff}. In view of Theorem \ref{CardonaRuzhansky} we conclude the proof.
\end{proof}

\bibliographystyle{amsplain}

\end{document}